\documentclass{amsart}

\usepackage{amssymb}
\usepackage{amsthm}
\usepackage{amsmath}

\usepackage[usenames]{color}

\usepackage{hyperref} 

\theoremstyle{plain} 
\newtheorem{proposition}{Proposition}[section]
\newtheorem{theorem}[proposition]{Theorem}
\newtheorem{lemma}[proposition]{Lemma}
\newtheorem{corollary}[proposition]{Corollary}
\theoremstyle{definition}

\newtheorem{definition}[proposition]{Definition}

\theoremstyle{remark}
\newtheorem{remark}[proposition]{Remark}

\numberwithin{equation}{section}

\DeclareMathOperator{\Gsf}{\mathsf{G}}
\DeclareMathOperator{\Psf}{\mathsf{P}}

\newcommand{\Asf}{\mathsf{A}}

\newcommand{\Nsf}{\mathsf{N}}
\newcommand{\Ksf}{\mathsf{K}}

\DeclareMathOperator{\Aut}{\mathsf{Aut}}
\DeclareMathOperator{\PGL}{\mathsf{PGL}}

\DeclareMathOperator{\Ac}{\mathcal{A}}
\DeclareMathOperator{\Cc}{\mathcal{C}}
\DeclareMathOperator{\Dc}{\mathcal{D}}
\DeclareMathOperator{\Fc}{\mathcal{F}}
\DeclareMathOperator{\Gc}{\mathcal{G}}
\DeclareMathOperator{\Qc}{\mathcal{Q}}
\DeclareMathOperator{\Oc}{\mathcal{O}}

\DeclareMathOperator{\Uc}{\mathcal{U}}

\DeclareMathOperator{\Pb}{\mathbb{P}}
\DeclareMathOperator{\Rb}{\mathbb{R}}
\DeclareMathOperator{\Nb}{\mathbb{N}}

\DeclareMathOperator{\mfg}{\mathfrak{g}}
\DeclareMathOperator{\mfa}{\mathfrak{a}}
\DeclareMathOperator{\mfp}{\mathfrak{p}}
\DeclareMathOperator{\mfk}{\mathfrak{k}}

\DeclareMathOperator{\dist}{d}
\DeclareMathOperator{\Gr}{Gr}

\newcommand{\opp}{\mathrm{i}}

\newcommand{\abs}[1]{\left|#1\right|}
\newcommand{\norm}[1]{\left\|#1\right\|}

\begin{document}

\title[Strict concavity of the growth indicator function]{Strict concavity of the growth indicator function for relatively Anosov groups}

\author[Kim]{Dongryul M. Kim}\email{dongryul.kim97@gmail.com}\address{Simons Laufer Mathematical Sciences Institute, USA}

\author[Oh]{Hee Oh}\email{hee.oh@yale.edu}\address{Department of Mathematics, Yale University, USA}

\author[Zimmer]{Andrew Zimmer}\email{amzimmer2@wisc.edu}\address{Department of Mathematics, University of Wisconsin-Madison, USA}

\date{\today}

\subjclass[2010]{}

\begin{abstract}
Let $\Gamma$ be a discrete subgroup of a connected semisimple real algebraic group of higher rank. The growth indicator function $\psi_\Gamma$ records the directional exponential growth of the Cartan projections of elements of $\Gamma$ in the positive Weyl chamber $\mathfrak a^+$. We prove that if $\Gamma$ is a non-elementary relatively Borel Anosov group, then $\psi_\Gamma$ is strictly concave on non-collinear directions. We prove this by establishing the $\mathcal C^1$-smoothness of the Manhattan hypersurface, defined as the unit level set of the critical-exponent map $\phi\mapsto\delta^\phi(\Gamma)$. More generally, for a non-elementary $\theta$-transverse group, we prove local $\mathcal C^1$-regularity near every point of the $\theta$-Manhattan hypersurface that is positive on the $\theta$-limit cone and has a critical gap at infinity. In particular, the $\theta$-Manhattan hypersurface is globally $\mathcal C^1$ for relatively $\theta$-Anosov groups, and their $\theta$-growth indicator  functions are strictly concave on non-collinear directions.
\end{abstract}

\maketitle

\section{Introduction}

Let $\Gsf$ be a connected semisimple real algebraic group.
The main result of this paper is a sufficient condition for the $\theta$-growth indicator function of a Zariski dense $\theta$-transverse group of $\Gsf$ to be strictly concave in a neighborhood of a point. This result implies that the $\theta$-growth indicator function of every Zariski dense relatively $\theta$-Anosov group of $\Gsf$ is strictly concave. We establish our main result by proving local $\Cc^1$-regularity of the $\theta$-Manhattan hypersurface of a non-elementary $\theta$-transverse group at every point that is positive on the
$\theta$-limit cone and has a critical gap at infinity.
 We begin by introducing the notation needed to state this result precisely.

Let $\Asf<\Gsf$ be a maximal real split torus, and set
$\mfa:=\operatorname{Lie} \Asf$. Fix a positive Weyl chamber
$\mfa^+\subset\mfa$, and let $\Delta\subset\mfa^*$ denote the
corresponding set of simple restricted roots. We denote by
$$
\kappa:\Gsf\to\mfa^+
$$
the Cartan projection.

Fix a nonempty subset $\theta\subset\Delta$ which is invariant under the opposition involution of $\frak a$ (see \eqref{opp}).
Consider the partial Cartan subspace
$$
\mfa_\theta
:=
\bigcap_{\alpha\in\Delta\smallsetminus\theta}\ker\alpha.
$$
Let
$$
\pi_\theta:\mfa\to\mfa_\theta
$$
be the canonical projection invariant under every Weyl group element
that fixes $\mfa_\theta$ pointwise, and set
$$
\kappa_\theta:=\pi_\theta\circ\kappa:\Gsf\to\mfa_\theta.
$$

A subgroup $\Gamma<\Gsf$ is called \emph{$\theta$-discrete} if the
restriction
$$
\kappa_\theta|_\Gamma:\Gamma\to\mfa_\theta
$$
is proper. Equivalently, for every $T>0$, the set
$
\left\{
\gamma\in\Gamma:
\norm{\kappa_\theta(\gamma)}\leq T
\right\}
$
is finite, where $\norm{\cdot}$ is any norm on $\mfa_\theta$.  While every
discrete subgroup of $\Gsf$ is $\Delta$-discrete, the $\theta$-discreteness hypothesis is required to introduce the following $\theta$-growth indicator function.
The $\theta$-growth indicator function of a $\theta$-discrete subgroup
$\Gamma$, denoted by
$$
\psi_\Gamma^\theta:
\mfa_\theta\to[0,+\infty)\cup\{-\infty\},
$$
records the directional exponential growth of the set
$\kappa_\theta(\Gamma)$. More precisely, for
$u\in\mfa_\theta\smallsetminus\{0\}$, define
$$
\psi_\Gamma^\theta(u)
:=
\norm{u}
\inf_{\Cc\ni u}
\limsup_{T\to+\infty}
\frac{1}{T}
\log\#\left\{
\gamma\in\Gamma:
\kappa_\theta(\gamma)\in\Cc,\ 
\norm{\kappa_\theta(\gamma)}\leq T
\right\},
$$
where the infimum is taken over all open cones
$\Cc\subset\mfa_\theta$ containing $u$, and set
$$
\psi_\Gamma^\theta(0):=0.
$$
This definition is independent of the choice of norm. When
$\theta=\Delta$, the growth indicator function was introduced by Quint
\cite{Quint_divergence}; the definition for general $\theta$ was
introduced by the first two named authors and Wang \cite{KOW_PD}. In rank
one, it reduces to the usual critical exponent.

The growth indicator function is positively homogeneous, concave, and upper
semicontinuous. Its effective support is the $\theta$-limit cone
$$
\mathcal L_\theta(\Gamma)
:=
\left\{
u\in\mfa_\theta:
t_n\kappa_\theta(\gamma_n)\to u
\text{ for some }
\{\gamma_n\}\subset\Gamma
\text{ and }
t_n\searrow0
\right\}.
$$
More precisely,
$$
\left\{
u\in\mfa_\theta:
\psi_\Gamma^\theta(u)\geq0
\right\}
=
\mathcal L_\theta(\Gamma),
$$
and $\psi_\Gamma^\theta=-\infty$ outside
$\mathcal L_\theta(\Gamma)$.

The principal class of $\theta$-discrete subgroups considered in this
paper consists of transverse groups. Let $\Psf_\theta<\Gsf$ be the
parabolic subgroup associated with $\theta$, and let
$$
\Fc_\theta:=\Gsf/\Psf_\theta
$$
be the corresponding partial flag manifold. We denote by
$$
\Lambda_\theta(\Gamma)\subset\Fc_\theta
$$
the $\theta$-limit set of $\Gamma$. Throughout the paper, every discrete subgroup $\Gamma$ is assumed to be
\emph{non-elementary} in the sense that $\Lambda_\theta(\Gamma)$ is infinite.
A  discrete subgroup $\Gamma<\Gsf$ is called
\emph{$\theta$-transverse} if the following two conditions hold:
\begin{itemize}
\item for every $\alpha\in\theta$ and every sequence of distinct
elements $\{\gamma_n\}\subset\Gamma$,
$$
\alpha(\kappa(\gamma_n))\to+\infty;
$$
\item any two distinct points $x,y\in\Lambda_\theta(\Gamma)$ are
transverse; equivalently, the diagonal $\Gsf$-orbit
$$
\Gsf\cdot(x,y)\subset\Fc_\theta\times\Fc_\theta
$$
is open.
\end{itemize}
For such a group, the action of $\Gamma$ on
$\Lambda_\theta(\Gamma)$ is a non-elementary convergence action.
There are two important subclasses of transverse groups. A
$\theta$-transverse group is called
\emph{$\theta$-Anosov} if its action on
$\Lambda_\theta(\Gamma)$ is a uniform convergence action, and
\emph{relatively $\theta$-Anosov} if this action is a
geometrically finite convergence action. These classes generalize,
respectively, convex cocompact and geometrically finite subgroups of
rank-one Lie groups.

\subsection{Strict concavity of the growth indicator function}

The main purpose of this paper is to prove the strict concavity of the
growth indicator function for relatively Anosov groups. Since
$\psi_\Gamma^\theta$ is positively homogeneous, it is linear along
each ray. Thus the strongest natural form of
concavity one can expect is strict concavity between non-collinear
directions.
We say that $\psi_\Gamma^\theta$ is \emph{strictly concave} if, whenever
$v,w\in\mfa_\theta$ are non-collinear and both
$\psi_\Gamma^\theta(v)$ and $\psi_\Gamma^\theta(w)$ are finite,
 one has
$$
\psi_\Gamma^\theta(v+w)
>
\psi_\Gamma^\theta(v)+\psi_\Gamma^\theta(w).
$$

For $\theta$-Anosov groups, work of Quint
\cite{Quint_Schottky} and
Sambarino \cite{Sambarino_report} (see also \cite{samb_hyper}, \cite{PS_eigenvalues}) implies
that $\psi_\Gamma^\theta$ is strictly concave, differentiable on the
relative interior of $\mathcal L_\theta(\Gamma)$, and has infinite
slope at every nonzero point of its relative boundary.
Their arguments rely on thermodynamic formalism (e.g. \cite{BCLS_gafa}) and do not directly
extend to relatively $\theta$-Anosov groups.

For relatively $\theta$-Anosov groups, the results of Canary, Zhang,
and the third named author, together with Quint's duality, yield the analogous
differentiability and infinite-slope properties
\cite{CZZ2025}. Thus strict concavity was the remaining missing part
of this regularity picture.

Our first main theorem fills this gap.

\begin{theorem}[Strict concavity of the growth indicator function]
\label{cor:GI}
Let $\Gamma<\Gsf$ be a Zariski dense relatively
$\theta$-Anosov group. Then its $\theta$-growth indicator function
$\psi_\Gamma^\theta$ is strictly concave.
\end{theorem}

\begin{remark}
Reyes--Wang \cite{RW2026} have independently established
Theorem~\ref{cor:GI}.
\end{remark}

We record two closely related consequences of
Theorem~\ref{cor:GI}. A linear form
$\phi\in\mfa_\theta^*$ is said to be \emph{tangent} to
$\psi_\Gamma^\theta$ at a nonzero vector
$u\in\frak a_\theta$ if
$$
\phi\geq\psi_\Gamma^\theta
\quad\text{and}\quad
\phi(u)=\psi_\Gamma^\theta(u).
$$
Since both $\phi$ and $\psi_\Gamma^\theta$ are positively
homogeneous, tangency at $u$ depends only on the ray
$\Rb_{>0}u$, which is necessarily in $\mathcal L_\theta(\Gamma)\smallsetminus\{0\}$.

\begin{corollary}[Tangent directions]
\label{cor:unique tangent direction}
Let $\Gamma<\Gsf$ be Zariski dense relatively
$\theta$-Anosov group. For every $\phi\in\mfa_\theta^*$, there is at
most one ray in
$\mathcal L_\theta(\Gamma)\smallsetminus\{0\}$ along which $\phi$ is
tangent to $\psi_\Gamma^\theta$.
In addition,  the tangent linear
forms in $\mfa_\theta^*$ are in one-to-one correspondence with the
rays contained in
$\operatorname{int}\mathcal L_\theta(\Gamma)$.
\end{corollary}

Indeed, strict concavity gives the uniqueness of the tangent ray.
When $\Gamma$ is Zariski dense, the limit cone
$\mathcal L_\theta(\Gamma)$ has nonempty interior in $\mfa_\theta$ \cite{Benoist_properties}.
The differentiability result of \cite{CZZ2025} associates a unique
tangent linear form to every ray in
$\operatorname{int}\mathcal L_\theta(\Gamma)$, while the
infinite slope result  ensures that the tangent ray of every tangent
linear form lies in this interior. Thus the correspondence is
bijective. The correspondence is given by
$$
\Rb_{>0}u
\mapsto
d\psi_\Gamma^\theta|_u.
$$

For a linear form $\phi\in\mfa_\theta^*$ that is positive on
$\mathcal L_\theta(\Gamma)\smallsetminus\{0\}$, this uniqueness has a
natural normalized formulation. The section
$$
S_\phi
:=
\left\{
u\in\mathcal L_\theta(\Gamma):
\phi(u)=1
\right\}
$$
is compact and convex. Since $\psi_\Gamma^\theta$ is upper
semicontinuous, it attains its maximum on $S_\phi$. Moreover,  strict concavity
implies that this maximum is attained at a unique point. We denote it
by $u_\phi$; thus
$$
\psi_\Gamma^\theta(u_\phi)
=
\max_{\substack{
u\in\mathcal L_\theta(\Gamma),
\phi(u)=1
}}
\psi_\Gamma^\theta(u).
$$ Thus $u_\phi$ is the
unique dominant asymptotic Cartan direction for the weighted counting
problem determined by $\phi$.

The uniqueness in Corollary~\ref{cor:unique tangent direction} was
needed in the local mixing result \cite{KOP_localmixing} to identify a direction arising from
thermodynamic formalism with the direction characterized by tangency
to the growth indicator function. This application was one of the main
motivations for the present work, together with the intrinsic
geometric question of whether growth indicator functions are strictly
concave.

\subsection{The Manhattan hypersurface}
Our approach to Theorem~\ref{cor:GI} is through a convex set dual to
the growth indicator function. For $\phi\in\mfa_\theta^*$, define the extended
critical exponent
$$
\delta^\phi(\Gamma)
:=
\inf\left\{
s>0:
\sum_{\gamma\in\Gamma}
e^{-s\phi(\kappa_\theta(\gamma))}
<+\infty
\right\}
\in[0,+\infty],
$$
where the infimum of the empty set is understood to be $+\infty$.

Whenever $\phi\circ\kappa_\theta$ is bounded below and proper on
$\Gamma$, this agrees with the counting exponent
$$
\delta^\phi(\Gamma)
=
\limsup_{T\to+\infty}
\frac{1}{T}
\log\#\left\{
\gamma\in\Gamma:
\phi(\kappa_\theta(\gamma))\leq T
\right\}.
$$
In particular, this applies when $\phi$ is positive on
$\mathcal L_\theta(\Gamma)\smallsetminus\{0\}$.

Following Burger \cite{Burger_Manhattan} and Quint
\cite{Quint_divergence,Quint_Schottky}, define
$$
\Qc_\theta(\Gamma)
:=
\left\{
\phi\in\mfa_\theta^*:
\delta^\phi(\Gamma)\leq1
\right\}.
$$
Its boundary
$$
\partial\Qc_\theta(\Gamma)
=
\left\{
\phi\in\mfa_\theta^*:
\delta^\phi(\Gamma)=1
\right\}
$$
is called the \emph{$\theta$-Manhattan hypersurface} of $\Gamma$.

Quint's duality identifies $\Qc_\theta(\Gamma)$ with the set of linear
forms dominating the growth indicator function:
$$
\Qc_\theta(\Gamma)
=
\left\{
\phi\in\mfa_\theta^*:
\phi \geq\psi_\Gamma^\theta
\right\}.
$$
Consequently,
$$
\partial\Qc_\theta(\Gamma)
=
\left\{
\phi\in\mfa_\theta^*: 
\text{$\phi$ is tangent to $\psi_\Gamma^\theta$}
\right\}.
$$

If $\phi$ is positive on
$\mathcal L_\theta(\Gamma)\smallsetminus\{0\}$, the same duality gives
the variational formula
$$
\delta^\phi(\Gamma)
=
\max_{
u\in\mathcal L_\theta(\Gamma), \phi(u)=1}
\psi_\Gamma^\theta(u).
$$
Thus the geometry of the Manhattan hypersurface is dual to the
geometry of the growth indicator function.

Our second main theorem establishes the regularity of this dual
hypersurface.

\begin{theorem}[Smoothness of the Manhattan hypersurface]
\label{thm:Manhattan rel Anosov}
Let $\Gamma<\Gsf$ be a non-elementary relatively
$\theta$-Anosov group. Then its $\theta$-Manhattan hypersurface
$\partial\Qc_\theta(\Gamma)$ is a $\Cc^1$ hypersurface.
\end{theorem}

\begin{remark}
Reyes--Wang \cite{RW2026} have independently established
Theorem~\ref{thm:Manhattan rel Anosov}.
\end{remark}

Theorem~\ref{thm:Manhattan rel Anosov} implies
Theorem~\ref{cor:GI} by Quint's duality \cite{Quint_Schottky}. 

\subsection{The general transverse case}

Theorem~\ref{thm:Manhattan rel Anosov} is a consequence of a local
regularity theorem for general transverse groups. The additional
hypothesis is formulated in terms of growth at infinity.

Motivated by the work of Pit--Schapira \cite{PS_finiteness} in
negative curvature, Wen \cite{Wen2026} introduced the
\emph{critical-gap-at-infinity} condition, also called the
\emph{strongly positively recurrent} condition, for transverse
groups. We recall its precise definition in
Section~\ref{subsec:critical gap}. Roughly speaking, this condition
requires that the exponential growth arising outside every compact
part of the associated flow space be strictly smaller than the full
exponential growth. Wen proved that a critical gap at infinity implies
 finiteness of the associated
Bowen--Margulis--Sullivan measure.

Our most general result is the following.

\begin{theorem}
\label{thm:transverse main}
Let $\Gamma<\Gsf$ be a non-elementary
$\theta$-transverse group. Suppose that
$$
\phi\in\partial\Qc_\theta(\Gamma)
$$
is positive on
$\mathcal L_\theta(\Gamma)\smallsetminus\{0\}$
and has a critical gap at infinity. Then
$\partial\Qc_\theta(\Gamma)$ is a $\Cc^1$ hypersurface in a
neighborhood of $\phi$.

If, in addition, $\Gamma$ is Zariski dense, then
$\partial\Qc_\theta(\Gamma)$ is strictly convex in a neighborhood of
$\phi$.
\end{theorem}

Canary, Zhang, and the third named author proved that, if $\Gamma$ is relatively
$\theta$-Anosov, then every
$
\phi\in\partial\Qc_\theta(\Gamma)
$
is positive on
$\mathcal L_\theta(\Gamma)\smallsetminus\{0\}$
\cite{CZZ2025}. Wen \cite{Wen2026} further proved that every such
$\phi$ has a critical gap at infinity. Therefore, for relatively
$\theta$-Anosov groups, Theorem~\ref{thm:transverse main} applies at every point of the
Manhattan hypersurface and yields
Theorem~\ref{thm:Manhattan rel Anosov}. 

\subsection{Comments on the proof}

The proof of Theorem~\ref{thm:transverse main} is inspired by the work
of Cantrell--Tanaka \cite{CT_Manhattan} on Manhattan curves associated
with pairs of premetrics on hyperbolic groups. Their argument gives
the $\Cc^1$-regularity of the Manhattan hypersurface for Anosov groups
and relies on the uniform hyperbolicity of the underlying group.

For relatively Anosov groups, this strategy can be adapted using the
reparametrizations of flow spaces developed in
\cite{ZZ_relatively,KO_entropy}. For general transverse groups,
however, such a reparametrization is unavailable because there is no
sufficiently controlled global structure outside compact subsets of
the flow space.

To overcome this, we construct  a bounded continuous map
$$
\mathbf f:\Gc\to\mfa_\theta,
$$
which we call the \emph{Cartan displacement observable}. Its orbit
integrals coarsely recover the partial Cartan displacement of
recurrent orbit segments. More precisely, for orbit segments whose
initial and terminal points lie over a fixed compact part of the
quotient,
$$
\int_0^T\mathbf f(g_tv)\,dt
$$
differs from the corresponding partial Cartan projection by a
uniformly bounded amount. The construction of the above observable is the main novelty of this paper.

This observable allows us to identify the derivative of a local
convex graph parametrizing the Manhattan hypersurface, at every point
where the derivative exists, with ratios of averages against the
corresponding Bowen--Margulis--Sullivan measure. As we will show, the
critical-gap-at-infinity condition is stable under small perturbations
of the linear form and gives uniform control of the mass of these
measures outside compact subsets. We use this to prove that the
relevant averages depend continuously on the linear form.

The local graphing function of the Manhattan hypersurface is convex and hence differentiable almost
everywhere. Its derivative agrees almost everywhere with a continuous
function obtained from the Bowen--Margulis--Sullivan averages. A
standard convexity argument then shows that the graphing function is
everywhere differentiable with continuous derivative, proving the
$\Cc^1$-regularity of the Manhattan hypersurface. In the Zariski dense
case, the local strict-convexity assertion follows from the known
rigidity of equality in the convexity inequality for critical
exponents.

\subsection{Notation}

Given two functions $f,g:X\to(0,+\infty)$, we write $f\ll g$, or equivalently $g\gg f$, if there exists a constant $C>0$ such that
$$
f(x)\leq Cg(x)
$$
for every $x\in X$. If $f\ll g$ and $g\ll f$, we write $f\asymp g$.

\subsection*{Acknowledgements} 
Oh is partially supported by National Science Foundation grant DMS-2450703. Zimmer was partially supported by a Sloan Research Fellowship and National Science Foundation grants DMS-2105580 and
DMS-2452068.  This material is based upon work supported by the National Science Foundation under Grant No. DMS-2424139, while the authors were in residence at the Simons Laufer Mathematical Sciences Institute in Berkeley, California, during the Spring 2026 semester.

\section{Semisimple algebraic groups} 
Throughout the paper, $\Gsf$ is a connected semisimple real algebraic group. Let $\mfg = \mfp + \mfk$ be a fixed Cartan decomposition of its Lie algebra where  $\mfk$ and $\mfp$ are $+1$ and $-1$ eigenspaces of a fixed Cartan involution of $\mfg$ respectively.
Let $\mfa \subset \mfp$ be a maximal abelian subspace, and fix a positive Weyl chamber $\mfa^+ \subset \mfa$. We denote by $\Sigma \subset \mfa^*$ the set of restricted roots and by $\Delta \subset \mfa^*$ the corresponding set of simple restricted roots corresponding to the choice of $\mfa^+$. Then we have the decomposition
$$
\mfg = \mfg_0 \oplus \bigoplus_{\alpha \in \Sigma} \mfg_\alpha
$$
where 
$$
\mfg_\alpha = \{ X \in \mfg : [H,X] = \alpha(H)X \text{ for all } H \in \mfa\}.
$$
Let $\Sigma^+$ (resp. $\Sigma^-$) denote the restricted roots which are nonnegative (respectively, nonpositive) linear combinations of elements of $\Delta$. 

\subsection{Cartan projection} Let $\Ksf < \Gsf$ denote the maximal compact subgroup with Lie algebra $\mfk$. By the Cartan decomposition,  for any $g \in \Gsf$ there exists a unique $\kappa(g) \in \mfa^+$ such that
$$
g \in \Ksf (\exp \kappa(g)) \Ksf.
$$
This defines the map $\kappa : \Gsf \to \mfa^+$, called the \emph{Cartan projection}.

Fix a representative $w_0 \in \Ksf$ of the longest Weyl element that has order two. The \emph{opposition involution} is the map \begin{equation}\label{opp}
\opp := - {\rm Ad}_{w_0} : \mfa \to \mfa.\end{equation} It satisfies 
$$
\kappa(g^{-1}) = \opp \kappa(g) \quad \text{for all} \quad g \in \Gsf. 
$$
The adjoint $\opp^* : \mfa^* \to \mfa^*$ of the opposition involution preserves the set of simple roots. For a subset $\theta \subset \Delta$, we define 
$$
\opp^* \theta := \{ \opp^*\alpha : \alpha \in \theta \}.
$$

\subsection{Parabolic subgroups and flag manifolds} Given a nonempty $\theta \subset \Delta$, the associated parabolic subgroup $\Psf_\theta$ is the stabilizer under the adjoint action of the Lie algebra 
$$
\mathfrak{u}_\theta^+ := \bigoplus_{\alpha \in \Sigma_\theta^+} \mfg_\alpha
$$
where $\Sigma_\theta^+ := \Sigma^+ \smallsetminus {\rm span}(\Delta \smallsetminus \theta)$. 

The \emph{Furstenberg boundary} and, more generally, the \emph{$\theta$-boundary} are the quotient spaces 
$$
\Fc_{\Delta} := \Gsf / \Psf_{\Delta} \quad \text{and} \quad \Fc_{\theta} := \Gsf / \Psf_{\theta}.
$$
We also call $\Fc_{\theta}$ a \emph{partial flag manifold}.
Two elements $x \in \Fc_\theta$ and $y \in \Fc_{\opp^* \theta}$ are \emph{transverse} if there exists $g \in \Gsf$ such that 
$$
x = g \Psf_\theta \quad \text{and} \quad y = g w_0 \Psf_{\opp^* \theta},
$$
Equivalently, the pair $(x,y)$ is contained in the unique open $\Gsf$-orbit in $\Fc_\theta \times \Fc_{\opp^* \theta}$.

\subsection{The partial Iwasawa cocycle} 
We denote by $\Nsf < \Psf_{\Delta}$ the unipotent radical of $\Psf_{\Delta}$. Then we have the Iwasawa decomposition $\Gsf = \Ksf \Asf \Nsf$, and the product map $\Ksf \times \Asf \times \Nsf \to \Gsf$ is a diffeomorphism.

The \emph{Iwasawa cocycle} $B_{\Delta} : \Gsf \times \Fc_{\Delta} \to \mfa$ is defined as follows: for $g \in \Gsf$ and $x \in \Fc_{\Delta}$, $B_{\Delta}(g, x) \in \mfa$ is the unique element such that
$$
gk \in \Ksf ( \exp B_{\Delta}(g, x)) \Nsf
$$
where $k \in \Ksf$ satisfies $k \Psf_{\Delta} = x$ in $\Fc_{\Delta}$.

For general $\theta \subset \Delta$, let 
$$
\mfa_\theta := \bigcap_{\alpha \notin \theta} \ker \alpha.
$$
For $\alpha \in \Delta$, let $\omega_\alpha$ denote the (restricted) fundamental weight associated with $\alpha$.
Then $\{ \omega_\alpha|_{\mfa_\theta}\}_{\alpha \in \theta}$ is a basis for the dual space $\mfa_\theta^*$. Thus, we can identify $\mfa^*_\theta$ with the span of $\{ \omega_\alpha\}_{\alpha \in \theta}$ and, in particular, view $\mfa_\theta^*$ as a subspace of $\mfa^*$. Moreover, there exists a unique projection 
\begin{equation*}
\pi_{\theta} : \mfa \to \mfa_{\theta}
\end{equation*}
satisfying 
\begin{equation}\label{eqn:defn of pi_theta to mfa_theta} 
\omega_\alpha(\pi_\theta(H)) = \omega_\alpha(H)\quad\text{for all $H \in \mfa$ and $\alpha \in \theta$. }
\end{equation}

The \emph{partial Iwasawa cocycle} $B_{\theta} : \Gsf \times \Fc_{\theta} \to \mfa_{\theta}$ is defined as 
\begin{equation}\label{eqn:defn of BIW} 
B_{\theta}(g, x) := \pi_{\theta}(B_{\Delta}(g, \tilde x))
\end{equation} 
for any $\tilde x \in \Fc_{\Delta}$ that projects to $x \in \Fc_{\theta}$ under the canonical projection $\Fc_{\Delta} \to \Fc_{\theta}$. The above definition is independent of the choice of $\tilde x$ \cite[Lemma 6.1]{Quint_PS}.

\subsection{The limit cone} \label{subsec:Benoist limit cone}
The \emph{limit cone} of $\Gamma < \Gsf$ is defined as the asymptotic cone of $\kappa(\Gamma) \subset \mfa^+$:
$$\mathcal{L}(\Gamma) := \{ X \in \mfa^+ : \exists \{\gamma_n\} \subset \Gamma \text{ and } t_n \searrow 0 \text{ such that } t_n \kappa(\gamma_n) \rightarrow X\}. 
$$

Denoting by $\kappa_{\theta} := \pi_{\theta} \circ \kappa : \Gsf \to \mfa_{\theta}^+$,  the \emph{$\theta$-limit cone}  $\mathcal{L}_{\theta}(\Gamma)$ is the asymptotic cone of $\kappa_{\theta}(\Gamma) \subset \mfa_{\theta}^+$.
We have
$$
\mathcal{L}_{\theta}(\Gamma) = \pi_{\theta} (\mathcal{L}(\Gamma)) .
$$

\subsection{Patterson--Sullivan measures}
Let $\Gamma<\Gsf$ be a non-elementary $\theta$-discrete subgroup. The
$\theta$-limit set $\Lambda_\theta(\Gamma)\subset\Fc_\theta$ consists
of all $\xi\in\Fc_\theta$ for which there exists a sequence of
distinct elements $\{\gamma_n\}\subset\Gamma$ such that
$$
\min_{\alpha\in\theta}\alpha(\kappa(\gamma_n))\to+\infty\quad\text{and}\quad 
k_n\Psf_\theta\to\xi,
$$
where
$
\gamma_n=k_n\exp(\kappa(\gamma_n))\ell_n$ with
$k_n,\ell_n\in\Ksf $. When $\Gamma$ is Zariski dense, this is equivalently the set of
$\xi\in\Fc_\theta$ for which
$(\gamma_n)_*\operatorname{Leb}_\theta$ converges to the Dirac mass
at $\xi$ for some sequence $\{\gamma_n\}\subset\Gamma$ where $\operatorname{Leb}_\theta$ denotes the $\Ksf$-invariant probability measure on $\mathcal F_\theta$. Moreover,
 $\Lambda_{\theta}(\Gamma) \subset \Fc_{\theta}$ is the unique closed $\Gamma$-minimal subset  of $\Fc_\theta$ \cite{Benoist_properties}. 

For  $\phi \in \mfa_{\theta}^*$ and $\delta \ge 0$, a Borel probability measure $\mu$ on $\mathcal F_\theta$ is called a \emph{$\delta$-dimensional $\phi$-Patterson--Sullivan measure} for $\Gamma$ on $\mathcal F_\theta$ if, for every $\gamma \in \Gamma$, the measure $\gamma_* \mu$ is absolutely continuous with respect to $\mu$ and  
$$
\frac{d \gamma_* \mu}{d \mu}(x) = e^{ - \delta \phi(B_{\theta}(\gamma^{-1}, x))} \quad \text{for } \mu \text{-a.e. } x \in \mathcal F_{\theta}.
$$

If $\phi\in \frak a_\theta^*$ is tangent to $\psi_\Gamma^\theta$
at an interior direction of $\frak a_\theta^+$, then there exists a one-dimensional $\phi$-Patterson--Sullivan measure for $\Gamma$ supported on $\Lambda_\theta(\Gamma)$ (\cite{Quint_PS} for $\theta=\Delta$, \cite[Proposition 5.12]{KOW_PD} in general).

 We call $\Gamma$ \emph{$\theta$-divergent} or $\theta$-regular if,  for any $\alpha \in \theta$ and any sequence $\{\gamma_n\} \subset \Gamma$ of distinct elements,
$$
\alpha(\kappa(\gamma_n)) \to + \infty \quad \text{as } n \to + \infty.
$$
For $\phi \in \mfa_{\theta}^*$, we set 
$$
\delta^{\phi}(\Gamma) := \limsup_{T \to + \infty} \frac{ \log \# \{ \gamma \in \Gamma : \phi(\kappa(\gamma)) \le T \}}{T} \in [0, + \infty].
$$
If $\phi>0$ on $\mathcal L_\theta-\{0\}$, then $0< \delta^{\phi}(\Gamma)<+\infty$ \cite[Lemmas 3.5, 4.3]{KOW_PD}.

\begin{lemma}  \label{obs:finite critical exp fund weight}
  If $\Gamma < \Gsf$ is $\{\alpha\}$-discrete for $\alpha\in \Delta$, then $0<\delta^{\omega_\alpha}(\Gamma) < + \infty$.
\end{lemma} 

\begin{proof}
Set $\theta_\alpha:=\{\alpha\}$. The space
$\mfa_{\theta_\alpha}$ is one-dimensional, and
$\mfa_{\theta_\alpha}^+$ is a ray.  Writing
$
\omega_\alpha
=
\sum_{\beta\in\Delta}c_\beta\beta,
$
the coefficient $c_\alpha$ is positive and $c_\beta\ge 0$. Since $\beta|_{\frak{a}_{\theta_\alpha}}=0$ for all $\beta\ne \alpha$,
$$
\omega_\alpha(H)
=
c_\alpha\alpha(H)>0 \quad\text{for every
$H\in\mfa_{\theta_\alpha}^+\smallsetminus\{0\}$.}
$$

Since $\mathcal L_{\theta_\alpha}(\Gamma)=\mfa_{\theta_\alpha}^+,$
it follows that $\omega_\alpha$ is positive on
$\mathcal L_{\theta_\alpha}(\Gamma)\smallsetminus\{0\}$. Therefore,
\cite[Lemmas 3.5, 4.3]{KOW_PD}, applied with
$\theta=\theta_\alpha$, gives
$
0<\delta^{\omega_\alpha}(\Gamma)<+\infty.
$
\end{proof}

In the following proposition, the linear form $\delta^{\phi}(\Gamma) \phi$ is always tangent to $\psi_\Gamma^\theta$ \cite[Lemma 4.5]{KOW_PD}, but the tangent direction is allowed to belong to the boundary of $\frak a_\theta^+$.
\begin{proposition}[{\cite[Proposition 3.2]{CZZ2024}}] \label{prop:PS existence}
  Suppose that $\Gamma < \Gsf$ is $\theta$-divergent. If $\phi \in \mfa_{\theta}^*$ satisfies $\delta^{\phi}(\Gamma) < + \infty$, then there exists a $\delta^{\phi}(\Gamma)$-dimensional $\phi$-Patterson--Sullivan measure for $\Gamma$ supported on $\Lambda_{\theta}(\Gamma)$.
\end{proposition}

Next, we recall some consequences of an ergodic dichotomy for transverse groups established in \cite{CZZ2024} and \cite{KOW_PD}. To state the result, we need a few definitions. Given $\phi \in \mfa_\theta^*$, the pair $(\Gamma, \phi)$ is of \emph{divergence type} if $\delta^\phi(\Gamma) < +\infty$ and 
$$
\sum_{\gamma \in \Gamma} e^{-\delta^{\phi}(\Gamma) \phi(\kappa(\gamma))} = +\infty.
$$
Also, recall that $\Gamma$ acts on $\Lambda_\theta(\Gamma)$ as a convergence group and hence there is a well-defined notion of conical limit point: $x \in \Lambda_\theta(\Gamma)$ is a \emph{conical limit point} if there exist a sequence $\{\gamma_n\} \subset \Gamma$ of distinct elements and distinct points $a,b \in \Lambda_\theta(\Gamma)$ such that $\gamma_n^{-1} x \rightarrow a$ and $\gamma_n^{-1} y \rightarrow b$ for all $y \in \Lambda_\theta(\Gamma) \smallsetminus \{x\}$. In this case, we say that \emph{$\{ \gamma_n \} \subset \Gamma$ converges conically to $x$}.
 The set of all conical limit points is called the \emph{conical limit set}.

\begin{theorem}[\cite{CZZ2024}, \cite{KOW_PD}]\label{thm:consequence of erg dich}    Suppose $\Gamma < \Gsf$ is a non-elementary $\theta$-transverse group and $\phi \in \mfa_{\theta}^*$ satisfies $\delta^{\phi}(\Gamma) < + \infty$. If $(\Gamma, \phi)$ is of divergence type, then:
\begin{enumerate}
\item There is a unique $\delta^{\phi}(\Gamma)$-dimensional $\phi$-Patterson--Sullivan measure for $\Gamma$ supported on $\Lambda_{\theta}(\Gamma)$; denote it by $\mu$.
\item If, in addition, $\Gamma$ is Zariski dense, then there is a unique $\delta^{\phi}(\Gamma)$-dimensional $\phi$-Patterson--Sullivan measure for $\Gamma$ supported on $\Fc_\theta$.
\item The action of $\Gamma$ on $(\Lambda_\theta(\Gamma), \mu)$ is ergodic. 
\item The conical limit set in $\Lambda_\theta(\Gamma)$ has full $\mu$-measure. 
\end{enumerate} 
\end{theorem}

\section{Projectively visible models of transverse groups}

In \cite{CZZ2024} and \cite{CZZ2026}, projectively visible models were constructed for transverse groups. In this section, we recall this construction.

\subsection{Background on convex real projective geometry}

We recall some terminology from real projective geometry.  

Suppose that $\Omega \subset \Pb(\Rb^d)$ is a \emph{properly convex domain}, i.e., $\Omega$ is a bounded convex open subset of some affine chart of $\Pb(\Rb^d)$. Its \emph{automorphism group} is
$$
\Aut(\Omega) := \{ g \in \PGL(d,\Rb) : g \Omega = \Omega\}.
$$
For a subgroup $H < \Aut(\Omega)$, its \emph{limit set} is the set $\Lambda_{\Omega}(H) \subset \partial \Omega$ of points $x \in \partial \Omega$ such that there exist $o \in \Omega$ and a sequence $\{h_n\} \subset H$ with $h_n o \rightarrow x$. 

Given $x,y \in \overline\Omega$, we let $[x,y]$ denote the projective line segment contained in $\overline\Omega$ joining $x$ to $y$. We further define $(x,y) :=[x,y] \smallsetminus \{x,y\}$, $[x,y) :=[x,y] \smallsetminus \{y\}$, and $(x,y] :=[x,y] \smallsetminus \{x\}$.

The \emph{Hilbert metric} $\dist_\Omega$ on $\Omega$ is defined as follows: for distinct $p,q \in \Omega$, let $a,b \in \partial \Omega$ be the unique points with $p,q \in (a,b)$ and in the order $a,p,q,b$, and define 
$$
\dist_\Omega(p,q) := \frac{1}{2} \log \frac{\norm{a-q}\norm{b-p}}{\norm{a-p}\norm{b-q}}
$$
where $\norm{\cdot}$ is any norm on any affine chart containing $\overline{\Omega}$. The Hilbert metric is a proper geodesic metric on which $\Aut(\Omega)$ acts by isometries. Moreover, for $p,q \in \Omega$, the line segment $[p,q]$ can be parametrized as a geodesic. 

Directly from the definition of the Hilbert metric and convexity, we have the following useful estimate on the distance between two geodesic rays with a common endpoint. 

\begin{lemma}\label{obs:distance between rays} Let $p,q \in \Omega$ and $x \in \partial \Omega$. For every $p' \in [p,x)$, we have
$$
\dist_\Omega(p', [q,x)) \leq \dist_\Omega(p,q). 
$$
\end{lemma}

An element $W \in \Gr_{d-1}(\Rb^d)$ 
is a \emph{supporting hyperplane} at $x \in \partial \Omega$ if $x \in \Pb(W)$ and $\Pb(W) \cap \Omega = \emptyset$. Convexity of $\Omega$ implies that every boundary point is contained in at least one supporting hyperplane, and we say that $\partial \Omega$ is \emph{$\Cc^1$-smooth at $x \in \partial\Omega$} if $x$ is contained in a unique supporting hyperplane.

We end this section by recording a result of Tholozan. Given a discrete group $\Gamma_0 < \Aut(\Omega)$, the \emph{Hilbert metric critical exponent} is 
$$
\delta_\Omega(\Gamma_0) := \limsup_{R \rightarrow  + \infty} \frac{ \log \#\{ \gamma \in \Gamma_0 : \dist_\Omega(\gamma o, o) \leq R\} }{R}
$$
where $o \in \Omega$ is any fixed point. It follows from Tholozan's work  that this critical exponent has a uniform upper bound, depending only on the dimension $d$. 

\begin{theorem}[\cite{Tholozan}]\label{thm:tholozan} If $\Gamma_0 < \Aut(\Omega)$ is discrete, then $\delta_\Omega(\Gamma_0) \leq d-2$. \end{theorem}

\subsection{Projectively visible groups} \label{sec:projectively visible} Next we define a special class of groups acting on properly convex domains, the projectively visible groups, and their relation with transverse groups.

\begin{definition} Let $\Omega \subset \Pb(\Rb^d)$ be a properly convex domain. A discrete subgroup $\Gamma_0 < \Aut(\Omega)$ is \emph{projectively visible} if 
\begin{itemize} 
\item $(x,y) \subset \Omega$ for every $x,y \in \Lambda_\Omega(\Gamma_0)$, 
\item every $x \in  \Lambda_\Omega(\Gamma_0)$ is a $\Cc^1$-smooth point of $\partial \Omega$. 
\end{itemize} 
\end{definition}

Under some mild conditions on $\Gsf$ and $\Psf_\theta$, every transverse group can be identified with a projectively visible subgroup. We fix any norm $\norm{\cdot}$ on $\mfa_\theta$. 

\begin{theorem}[{\cite[Theorem 6.2]{CZZ2024}}] \label{thm:transverse have proj models} Suppose that $\Gsf$ has trivial center, that $\theta \subset \Delta$ is a nonempty subset with $\theta = \opp^* \theta$,  
   and that $\Psf_\theta$ contains no simple factors of $\Gsf$. If $\Gamma < \Gsf$ is $\theta$-transverse, then there exist $d \in \Nb$, a properly convex domain $\Omega \subset \Pb(\Rb^d)$, a projectively visible subgroup $\Gamma_0 < \Aut(\Omega)$, an isomorphism $\rho : \Gamma_0 \rightarrow \Gamma$, and a $\rho$-equivariant homeomorphism $\xi : \Lambda_\Omega(\Gamma_0) \rightarrow \Lambda_\theta(\Gamma)$. 
   
   Moreover, if $o \in \Omega$, then there exists $C_0 > 1$ such that  for all $\gamma \in \Gamma_0$,
   $$
   C_0^{-1} \dist_\Omega(o,\gamma o) -C_0 \leq \norm{\kappa_\theta(\rho(\gamma))}\leq    C_0 \dist_\Omega(o,\gamma o) + C_0.
   $$
  
\end{theorem}

\begin{proof} The ``moreover'' assertion is not stated explicitly in \cite[Theorem 6.2]{CZZ2024} so we briefly explain why it follows from the proof in \cite[Appendix B]{CZZ2024}. In particular, there exists an irreducible faithful representation $\Phi : \Gsf \rightarrow \mathsf{PSL}(d,\Rb)$ such that $\Phi \circ \rho  = {\rm id}_{\Gamma_0}$. Further, there exist positive integers $\{m_\alpha\}_{\alpha \in \theta}$ such that 
$$
\log \norm{\Phi(g)} = \sum_{\alpha \in \theta} m_\alpha \omega_\alpha(\kappa_\theta(g))
$$
for all $g \in \Gsf$, where $ \norm{\Phi(g)} $ is the operator norm of $\Phi(g)$. 
Hence
$$
\log \left( \norm{\Phi(g)}\norm{\Phi(g)^{-1}}\right)= \sum_{\alpha \in \theta} (m_\alpha+m_{\opp^*\alpha}) \omega_\alpha(\kappa_\theta(g))
$$
for all $g \in \Gsf$. By \cite[Proposition 10.1]{DGK2024}, there exists $c_1 > 0$ such that 
$$
\abs{ \dist_\Omega(o,\gamma o) - \log \left( \norm{\Phi(\rho(\gamma))}\norm{\Phi(\rho(\gamma))^{-1}} \right)} \leq c_1
$$
for all $\gamma \in \Gamma_0$. Since $\sum_{\alpha \in \theta} (m_\alpha+m_{\opp^*\alpha})  \omega_\alpha$ is positive on $\mfa_\theta^+ \smallsetminus \{0\}$, there exists a constant $c_2> 1$ such that for all $g \in \Gsf$,
$$
c_2^{-1}\norm{\kappa_\theta(g)} \leq  \sum_{\alpha \in \theta} (m_\alpha+m_{\opp^*\alpha}) \omega_\alpha(\kappa_\theta(g))\leq c_2\norm{\kappa_\theta(g)} .$$
 Finally, using the fact that $\Phi \circ \rho  = {\rm id}_{\Gamma_0}$, we obtain  for all $\gamma \in \Gamma_0$,
   $$
   c_2^{-1} \dist_\Omega(o,\gamma o)  - c_2^{-1} c_1 \leq \norm{\kappa_\theta(\rho(\gamma))}\leq    c_2\dist_\Omega(o,\gamma o)  + c_2 c_1. 
   $$ \end{proof} 
 
 For general $\Gsf$, by replacing $\Gsf$ with a quotient, one can always assume that  $\Gsf$ has trivial center   and $\Psf_\theta$ contains no simple factors of $\Gsf$, see \cite[Section 2.4]{CZZ2024}. Hence, we can always find $(\Omega, \Gamma_0, \rho, \xi)$ satisfying Theorem~\ref{thm:transverse have proj models}.

\begin{definition} Given a $\theta$-transverse group $\Gamma < \Gsf$, a \emph{projectively visible model} of $\Gamma$ is a choice of $(\Omega, \Gamma_0, \rho, \xi)$ that satisfies the conclusion of Theorem~\ref{thm:transverse have proj models}.

\end{definition}

\subsection{Shadows}
We assume $\theta = \opp^* \theta$. 
Suppose that $\Gamma < \Gsf$ is a non-elementary $\theta$-transverse group. Choose a  projectively visible model   $(\Omega, \Gamma_0, \rho, \xi)$ of $\Gamma$. Since $\xi : \Lambda_{\Omega}(\Gamma_0) \to \Lambda_{\theta}(\Gamma)$ is a $\rho$-equivariant homeomorphism, we may regard measures on $\Lambda_{\theta}(\Gamma)$ as measures on $\Lambda_{\Omega}(\Gamma_0)$ by pulling them back through $\xi$.

Fix a basepoint $o \in \Omega$. For $R > 0$ and $\gamma \in \Gamma_0$, we define the shadow of the Hilbert $R$-ball centered at $\gamma o$ by
$$
\Oc_R(\gamma) := \{ x \in \partial \Omega : \dist_{\Omega}(\gamma o, [o, x)) < R \}.
$$
We also write 
$$\widehat{\Oc}_R(\rho(\gamma)) := \xi \left( \Oc_R(\gamma) \cap \Lambda_{\Omega}(\Gamma_0) \right) \subset \Lambda_{\theta}(\Gamma).$$

The following shadow Lemma  estimates the size of a shadow with respect to a Patterson--Sullivan measure.

\begin{proposition}[Shadow Lemma, {\cite[Proposition 7.1]{CZZ2024}}] \label{prop:Shadow Lemma}
Let $\phi \in \mfa_{\theta}^*$ and $\delta \ge 0$. Let $\mu$ be a $\delta$-dimensional $\phi$-Patterson--Sullivan measure for $\Gamma$ on $\Lambda_{\theta}(\Gamma)$. Then for all sufficiently large $R > 0$, there exists $C = C(R) > 1$ such that
   for all $\gamma \in \Gamma$, $$
  C^{-1} e^{-\delta \phi(\kappa(\gamma))} \le \mu \left( \widehat{\Oc}_R(\gamma) \right) \le C  e^{-\delta \phi(\kappa(\gamma))}.
  $$

\end{proposition}

An important property of shadows is their uniform multiplicity. Fix a linear form $\phi \in \mfa_{\theta}^*$ such that there exists a $\delta$-dimensional $\phi$-Patterson--Sullivan measure for $\Gamma$ on $\Lambda_{\theta}(\Gamma)$. Then for $n \in \Nb$, we set
$$
S_n^{\phi} := \{ \gamma \in \Gamma : \phi(\kappa(\gamma)) \in [n, n+1) \}.
$$

\begin{lemma}[Uniform multiplicity, {\cite[Lemma 8.2]{CZZ2024}}] \label{lem:uniform multiplicity}
Suppose $\phi \in \mfa_{\theta}^*$ is as above. Then 
  for any $R > 0$,
  $$
  \sup_{n \in \Nb} \sup_{x \in \Lambda_{\theta}(\Gamma)} \# \left\{ \gamma \in S_n^{\phi} : x \in \widehat{\Oc}_R(\gamma) \right\} < + \infty.
  $$
\end{lemma}

\subsection{Conical limit points}

We continue to suppose that $\Gamma < \Gsf$ is a non-elementary $\theta$-transverse group, $(\Omega, \Gamma_0, \rho, \xi)$ is a projectively visible model  of $\Gamma$, and $o \in \Omega$ is a fixed basepoint.  

Recall that the convergence action of $\Gamma$ on $\Lambda_\theta(\Gamma)$ has a well-defined notion of conical limit point and hence a conical limit set. The action of $\Gamma_0$ on $\Omega$ gives an equivalent characterization in terms of geodesic rays. The next result records this equivalence. 

\begin{proposition} \label{prop:char of conical} If $x \in \Lambda_\Omega(\Gamma_0)$ and $\{\gamma_n\} \subset \Gamma_0$ is a sequence of distinct elements, then the following are equivalent:  
  \begin{enumerate}
    \item $\{\rho(\gamma_n)\} \subset \Gamma$ converges conically to $\xi(x) \in \Lambda_\theta(\Gamma)$ (in the convergence group sense).
\item There exist $p \in \Omega$ and $R > 0$ such that $\gamma_n p \rightarrow x$ and 
$$
\sup_{n \geq 1} \dist_\Omega( \gamma_n p, [p, x)) < R. 
$$
\item There exist $R > 0$ such that 
$$
x \in \bigcap_{n \geq 1} \Oc_R(\gamma_n).
$$
  \end{enumerate}

\end{proposition} 

\begin{proof} 
  The equivalence between (1) and (2) is due to \cite[Lemma 3.6]{CZZ2026}. Thus it suffices to show that (2) and (3) are equivalent.  

Suppose (2) holds. Let $R_1 := \dist_\Omega(p,o)$. For each $n$, choose $p_n' \in [p,x)$ with $\dist_\Omega(\gamma_n p,p_n') < R$. Then, by Lemma~\ref{obs:distance between rays},
$$
\dist_\Omega(p_n',[o,x)) \leq \dist_\Omega(p,o) = R_1.
$$
Therefore,
$$
\dist_\Omega(\gamma_n o,[o,x)) \leq \dist_\Omega(\gamma_n o,\gamma_n p)+\dist_\Omega(\gamma_n p,p_n')+\dist_\Omega(p_n',[o,x)) < R+2R_1. 
$$
Thus 
$$
x \in \bigcap_{n \geq 1} \Oc_{R+2R_1}(\gamma_n)
$$
and hence (3) holds. 

Suppose (3) holds. Then clearly
$$
\sup_{n \geq 1} \dist_\Omega( \gamma_n o, [o, x)) < R. 
$$
and so we just need to prove that $\gamma_n o \rightarrow x$. Let $\dist_{\Pb}$ be a distance on $\Pb(\Rb^d)$ induced by a Riemannian metric. By \cite[Lemma 12.5]{CZZ2024}, we have ${\rm diam} \Oc_R(\gamma_n) \rightarrow 0$ with respect to this distance. For each $n$, we can fix $x_n \in \partial \Omega$ such that $\gamma_n o \in [o ,x_n]$. Then $x_n \in \Oc_R(\gamma_n) $. Since $\{\gamma_n\}$ are distinct and $\Gamma_0$ acts properly on $\Omega$, we have $\dist_{\Pb}(x_n, \gamma_n o) \rightarrow 0$. Since $x, x_n \in \Oc_R(\gamma_n)$ and ${\rm diam} \Oc_R(\gamma_n) \rightarrow 0$ we have $\dist_{\Pb}(x_n, x) \rightarrow~0$. Thus $\gamma_n o \rightarrow x$. 
\end{proof}

\section{Critical gap at infinity}

Suppose that $\theta = \opp^* \theta$, that $\Gamma < \Gsf$ is a non-elementary $\theta$-transverse group, and that $(\Omega, \Gamma_0, \rho, \xi)$ is a projectively visible model  of $\Gamma$. 

In this section, we consider the flow space given by the projectively visible model and Bowen--Margulis--Sullivan measures on it, and discuss their relation to the critical gap at infinity, following Wen's work \cite{Wen2026}.

\subsection{Flow spaces} 
We denote by $T^1 \Omega$ the unit tangent bundle over $\Omega$ with respect to the infinitesimal Hilbert metric, and let $g_t : T^1 \Omega \rightarrow  T^1 \Omega$, $t \in \Rb$, denote the geodesic flow.
For $v \in T^1 \Omega$, let $v^{\pm} = \lim_{t \to \pm \infty} g_t v \in \partial \Omega$ be the forward and backward endpoints of $v$, respectively.

We consider the flow space
$$\Gc := \{ v \in T^1 \Omega : v^+, v^- \in \Lambda_\Omega(\Gamma_0) \}.$$
 Let $\Lambda_\Omega^{(2)}(\Gamma_0)\subset \Lambda_\Omega(\Gamma_0) \times \Lambda_\Omega(\Gamma_0)$ denote the set of ordered pairs of distinct points. Since $\Gamma_0$ is projectively visible, $\Gc$ is homeomorphic to $\Lambda_{\Omega}^{(2)}(\Gamma_0) \times \Rb$. Using the Hopf parametrization for $\Gc$, we can make this identification explicit.

Since each point in $\Lambda_\Omega(\Gamma_0)$ is $\Cc^1$-smooth, for $\zeta \in \Lambda_\Omega(\Gamma_0)$ and $a,b \in \Omega$, the following limit exists (see \cite[Lemma 3.2]{Bray_ergodicity}):
$$
\beta_\zeta(a,b) := \lim_{\Omega \ni p \rightarrow \zeta} \dist_\Omega(p,a) - \dist_\Omega(p,b).
$$
Then fixing a base point $o \in \Omega$, the map 
$$
v \in \Gc \mapsto (v^-, v^+, \beta_{v^+}(o, \pi(v))) 
$$
is a homeomorphism and the natural $\Gamma_0$-action on $\Gc$ corresponds to the $\Gamma_0$-action on $\Lambda^{(2)}_\Omega(\Gamma_0) \times \Rb$ given by 
$$
\gamma \cdot (x,y,t) = (\gamma x, \gamma y, t + \beta_y(\gamma^{-1}o, o)). 
$$
For the rest of the paper, we identify 
$$
\Gc = \Lambda_\Omega^{(2)}(\Gamma_0) \times \Rb
$$
using this map. Under this identification, the Hilbert geodesic flow is given by 
$$
g_t(x,y,s) = (x,y,s+t) \quad \text{for } t \in \Rb.
$$

\subsection{BMS measures}

Suppose $\phi \in \mfa_\theta^*$ and $(\Gamma, \phi)$ is of divergence type. We define the Bowen--Margulis--Sullivan (BMS) measure on $\Gc$ associated with $\phi$. 

For $x, y \in \Lambda_{\theta}(\Gamma)$, the $\mfa_{\theta}$-valued Gromov product is defined as 
$$
\langle x, y \rangle_{\theta} := - \left( B_{\theta}(g^{-1}, x) + \opp B_{\theta}(g^{-1}, y) \right)
$$
for $g \in \Gsf$ with $g \Psf_{\theta} = x$ and $g w_0 \Psf_{\theta} = y$, noting that we are assuming $\theta = \opp^* \theta$. This is well defined and independent of the choice of $g$.

Notice that $\delta := \delta^{\phi}(\Gamma)= \delta^{\opp^* \phi}(\Gamma)$. So by Theorem~\ref{thm:consequence of erg dich}, there exist unique $\delta$-dimensional $\phi$- and $\opp^*\phi$-Patterson--Sullivan measures $\mu$ and $\mu_{\opp}$ for $\Gamma$ on $\Lambda_{\theta}(\Gamma)$, respectively. Then the Radon measure on $\Lambda_{\theta}^{(2)}(\Gamma)$ defined by 
$$
e^{ \delta \phi\left( \langle x, y \rangle_{\theta}\right)} d \mu(x) d \mu_{\opp}(y)
$$
is $\Gamma$-invariant. Hence, the Radon measure $\tilde m_{\phi}$ on $\Gc = \Lambda_{\Omega}^{(2)}(\Gamma_0) \times \Rb$ defined by
$$
d \tilde m_{\phi}(v) := e^{ \delta \phi\left( \langle \xi(v^+), \xi(v^-) \rangle_{\theta}\right)} d \mu(\xi(v^+)) d \mu_{\opp}(\xi(v^-)) dt
$$
is invariant under the $\Gamma_0$-action and the  Hilbert geodesic flow, where $dt$ denotes the Lebesgue measure on the $\Rb$-component. Since $\Gamma_0$ acts properly on $\Gc$, this measure descends to a geodesic-flow-invariant measure $m_{\phi}$ on the quotient $\Gamma_0 \backslash \Gc$ called the \emph{BMS measure associated with $\phi$}. 

This measure has the following dynamical properties. 

\begin{theorem}[{\cite{CZZ2024, KOW_PD,BCZZ2024}}] \label{thm:consequence of finite volume} The geodesic flow on $\Gamma_0 \backslash \Gc$ is ergodic with respect to $m_\phi$. Moreover, if $\norm{m_\phi} < +\infty$, then the geodesic flow on $\Gamma_0 \backslash \Gc$ is mixing with respect to $m_\phi$. 
\end{theorem}

\begin{proof} 
Ergodicity follows from the ergodic dichotomy in \cite{CZZ2024,KOW_PD}, as we are assuming that $(\Gamma, \phi)$ is of divergence type. The Hilbert length spectrum $\{ \ell_{\Omega}(\gamma) : \gamma \in \Gamma_0 \}$, where $\ell_{\Omega}$ denotes the Hilbert translation length, generates a dense additive subgroup of $\Rb$ by \cite[Lemma 2.8]{BZ2023}. Then mixing follows from the proof of mixing for continuous GPS systems in \cite[Theorem 4.2]{BCZZ2024}.
\end{proof}

\subsection{Critical gap at infinity}  \label{subsec:critical gap}

In this section, we recall Wen's results in \cite{Wen2026} showing that a ``critical gap at infinity'' implies the finiteness of the BMS measure.

Suppose that $\theta $ is $\operatorname{i}$-invariant and that $\Gamma < \Gsf$ is a non-elementary $\theta$-transverse group.
Fix $\phi \in \mfa_\theta^*$. We suppose that  $(\Gamma, \phi)$ is of divergence type, and that $(\Omega, \Gamma_0, \rho, \xi)$ is a projectively visible model  of $\Gamma$.

For a compact subset $K \subset \Gc$, let $\Gamma_K \subset \Gamma_0$ be the subset of elements $\gamma \in \Gamma_0$ for which there exist $v \in K$ and $t \geq 0$ such that  
$$
g_t v \in \gamma K \quad \text{and} \quad g_{[0,t]}v \cap \Gamma_0  K \subset K \cup \gamma K.
$$
Then for $\phi \in \mfa_\theta^*$, let
$$
\delta^\phi(K) := \inf \left\{ s > 0 : \sum_{\gamma \in \Gamma_K} e^{-s \phi(\kappa(\rho(\gamma)))} < +\infty\right\}.
$$

 By \cite[Lemma 4.2]{Wen2026}, there exists a compact set $K_0 \subset \Gc$ such that if $\gamma \in \Gamma_0$, then  $g_{[0,\infty)}v \cap \gamma K_0 \neq \emptyset$ for some $v \in K_0$. For subsets containing $K_0$, Wen established the following monotonicity property. 
 
 \begin{lemma}[{\cite[Lemma 4.4]{Wen2026}}] \label{lem:monotone of entropy of K} Suppose $K_1, K_2 \subset \Gc$ are compact. If $K_0 \subset K_1 \subset {\rm int}(K_2)$, then 
 $$
 \delta^\phi(K_2) \leq \delta^\phi(K_1).
 $$
 \end{lemma} 
 
 The \emph{entropy at infinity} for $\phi \in \mfa_\theta^*$ is
 $$
 \delta_\infty^\phi(\Gamma) := \inf\left\{ \delta^\phi(K) : K \subset \Gc \text{ compact and } K_0 \subset K\right\}.
 $$
Then $\phi$ has a \emph{critical gap at infinity for $\Gamma$} if 
 $$
  \delta_\infty^\phi(\Gamma) <  \delta^\phi(\Gamma).
  $$
  This is also referred to as the \emph{strongly positively recurrent} (SPR) condition.
  
  Recall the $\theta$-limit cone $\mathcal{L}_{\theta}(\Gamma)$ from Section~\ref{subsec:Benoist limit cone}.
  \begin{theorem}[{\cite[Theorems 4.8 and 4.9]{Wen2026}}] \label{thm:Wen's result} If $\phi \in \mfa_\theta^*$ is positive on 
    $\mathcal{L}_{\theta}(\Gamma) \smallsetminus \{0\}$ and has a critical gap at infinity for $\Gamma$, then $(\Gamma, \phi)$ is of divergence type and the  BMS measure on $\Gamma_0 \backslash \Gc$ associated with $\phi$ is finite. 
  \end{theorem} 
  
 \begin{proof} Wen's results are stated in the context of abstract GPS systems, and in Appendix~\ref{sec:appendix on Wen}, we state Wen's general result and explain why it implies the above theorem. \end{proof}

\section{Flow observables encoding Cartan displacement} \label{sec:constructing a function}
In this section,
we continue to assume that $\theta = \opp^* \theta$, that $\Gamma < \Gsf$ is a non-elementary $\theta$-transverse group, and that $(\Omega, \Gamma_0, \rho, \xi)$ is a projectively visible model  of $\Gamma$. The main construction of the paper is given  in Proposition~\ref{prop:constructing functions}, where we construct Cartan displacement observables on $\Gc$ whose orbit integrals
coarsely recode the Cartan displacement of recurrent geodesic flow. These observables will be a key ingredient in the proof of Theorem~\ref{thm:transverse main}.

\begin{proposition}[Cartan displacement observable] \label{prop:constructing functions}
There exists a $\Gamma_0$-invariant continuous bounded map
$$
\mathbf{f}:\Gc\to \mfa_\theta,
$$
 with the following
property:
For every compact set $K\subset\Gc$, there exists $C_K>0$ such
that, if $v\in K$, $\gamma\in\Gamma_0$, and
$g_Tv\in\gamma K$ for some $T\geq0$, then
$$
\left\|
\int_0^T\mathbf{f}(g_tv)\,dt
-
\kappa_\theta(\rho(\gamma))
\right\|
\leq C_K.
$$
\end{proposition} 
The associated additive Cartan displacement cocycle is
$$
\mathbf{F}(v,T)
:=
\int_0^T\mathbf{f}(g_tv)\,dt.
$$
It satisfies
$$
\mathbf{F}(v,S+T)
=
\mathbf{F}(v,S)+\mathbf{F}(g_Sv,T),
$$
and Proposition~\ref{prop:constructing functions} says that
$\mathbf{F}(v,T)$ coarsely records the partial Cartan displacement of
every orbit segment whose initial and terminal points lie over a fixed
compact part of the quotient.

For $\phi\in\mfa_\theta^*$, define the associated
\emph{$\phi$-Cartan displacement observable} by
$$
f_\phi:=\phi\circ\mathbf{f}.
$$
Then the map $(\phi,v)\mapsto f_\phi(v)$ is continuous, the map
$\phi\mapsto f_\phi$ is linear, and
$$
\sup_{v\in\Gc}\abs{f_\phi(v)}
\leq C_0\norm{\phi}
$$ where $C_0=\sup_{v\in\Gc}\|\mathbf{f}(v)\|$.
Moreover,
for every compact set $K\subset\Gc$, 
\begin{equation}\label{fphi}
\left|
\int_0^T f_\phi(g_tv)\,dt
-
\phi\bigl(\kappa_\theta(\rho(\gamma))\bigr)
\right|
\leq C_K\norm{\phi}
\end{equation}
whenever $v\in K$ and $g_Tv\in\gamma K$.

The function $f_\phi$ descends to a bounded continuous function
$$\hat f_\phi : \Gamma_0 \backslash \Gc \rightarrow \Rb.$$

Before proving Proposition~\ref{prop:constructing functions}, we give one application. 
Fix $\psi \in \mfa_{\theta}^*$ such that $(\Gamma,\psi)$ is of divergence type and $\norm{m_\psi}<+\infty$. Let $\mu_\psi$ denote the unique  $\delta^\psi(\Gamma)$-dimensional $\psi$-Patterson--Sullivan measure for $\Gamma$ on $\Lambda_\theta(\Gamma)$, given by Theorem~\ref{thm:consequence of erg dich}.

\begin{corollary} \label{cor:limit of ratios}
There exists a set $A \subset \Lambda_\theta(\Gamma)$ with full $\mu_\psi$-measure with the following property: if $x \in A$, $\{\gamma_n\} \subset \Gamma$ converges to $x$ conically, $\phi_1,\phi_2 \in \mfa_\theta^*$, and $\int \hat f_{\phi_2}\,dm_\psi \neq 0$, then 
$$
\lim_{n \rightarrow \infty} \frac{\phi_1(\kappa(\gamma_n))}{\phi_2(\kappa(\gamma_n))} = \frac{\int \hat f_{\phi_1} dm_\psi}{\int \hat f_{\phi_2} dm_\psi}.
$$
\end{corollary} 

\begin{proof} By Theorem~\ref{thm:consequence of finite volume}, the geodesic flow on $(\Gamma_0 \backslash \Gc, m_\psi)$ is ergodic. Fix a set $A_0 \subset \Gamma_0 \backslash \Gc$ of full $m_\psi$-measure such that the Birkhoff ergodic theorem holds on $A_0$ for each function in the finite set $\{ \hat f_{\omega_\alpha}\}_{\alpha \in \theta}$. By linearity of $\phi \mapsto \hat f_\phi$,  the Birkhoff ergodic theorem holds on $A_0$ for each function in $\{ \hat f_\phi\}_{\phi \in \mfa_\theta^*}$.

Let $A_0' \subset \Gc$ denote the lift of $A_0$. Then $A_0'$ has full $\tilde m_\psi$-measure and hence 
$$
A := \{ x \in \Lambda_\theta(\Gamma) : x = \xi(v^+) \text{ for some } v \in A_0'\}
$$
has full $\mu_\psi$-measure. 

Now fix $x \in A$, a sequence $\{\gamma_n\} \subset \Gamma$ converging conically to $x$, and $\phi_1,\phi_2 \in \mfa_\theta^*$ such that $\int \hat f_{\phi_2}\,dm_\psi \neq 0$. 

Fix $v \in A_0'$ such that $x = \xi(v^+)$. Let $\eta_n := \rho^{-1}(\gamma_n)$. Since $\xi : \Lambda_\Omega(\Gamma_0) \rightarrow \Lambda_\theta(\Gamma)$ is a homeomorphism, $\{\eta_n\} \subset \Gamma_0$ converges conically to $v^+$. Then, by Proposition~\ref{prop:char of conical}, there exist a compact subset $K \subset \Gc$ and a sequence $\{t_n\} \subset [0,+\infty)$ with $t_n \to +\infty$ such that $g_{t_n} v \in \eta_n K$ for all $n \geq 1$. Then, by the Birkhoff ergodic theorem and \eqref{fphi},
\begin{equation*}
\frac{\int \hat f_{\phi_1} dm_\psi}{\int \hat f_{\phi_2} dm_\psi}=\lim_{n \rightarrow \infty} \frac{ \int_0^{t_n} f_{\phi_1}(g_tv) dt}{ \int_0^{t_n} f_{\phi_2}(g_{t}v) dt}
=\lim_{n \rightarrow \infty} \frac{\phi_1(\kappa(\gamma_n))}{\phi_2(\kappa(\gamma_n))}. \qedhere \end{equation*} 
\end{proof} 

\begin{remark} 
We remark that when $\norm{m_{\psi}} = + \infty$, one may try to apply the Hopf ratio ergodic theorem to prove a similar result. However, in this case, $\hat f_{\phi}$ may not be in $L^1(\Gamma_0 \backslash \Gc, m_{\psi})$ and hence the Hopf ratio ergodic theorem may not apply.
\end{remark}

Next, we verify that linear forms that are positive on the $\theta$-limit cone have positive integrals. 

\begin{lemma} \label{obs:positive integral}
   If $\phi \in \mfa_\theta^*$ is positive on $\mathcal{L}_{\theta}(\Gamma) \smallsetminus \{0\}$,  then 
$$
\int \hat f_{\phi} dm_\psi > 0.
$$
\end{lemma}  

\begin{proof} Since $\mu_\psi$-a.e.\ point in $\Lambda_\theta(\Gamma)$ is conical (see Theorem~\ref{thm:consequence of erg dich}), arguing as in the preceding proof, we can fix $v \in \Gc$ such that 
$$ 
\lim_{T \rightarrow + \infty} \frac{1}{T} \int_0^T f_\phi(g_t v) dt = \frac{1}{\norm{m_{\psi}}} \int \hat f_{\phi} dm_\psi 
$$
and there exist a compact subset $K \subset \Gc$, a sequence $\{t_n\} \subset [0,+ \infty)$ with $t_n \to + \infty$, and a sequence $\{\eta_n\} \subset \Gamma_0$ such that $g_{t_n} v \in \eta_n K$ for all $n \geq 1$. Then there exists $C_1 > 0$ such that 
$$
\abs{ \dist_\Omega(o,\eta_n o)-t_n} \leq C_1. 
$$
Then by \eqref{fphi}, 
$$
\lim_{n \rightarrow \infty}  \frac{\phi(\kappa(\rho(\eta_n)))}{ \dist_\Omega(o,\eta_n o)}= \frac{1}{\norm{m_{\psi}}}  \int \hat f_{\phi} dm_\psi.
$$
Since $\phi$ is positive on $\mathcal{L}_{\theta}(\Gamma) \smallsetminus \{0\}$,  there exists $C_2 > 1$ such that 
$$
\phi(\kappa_\theta(\rho(\eta_n))) \geq C_2^{-1} \norm{\kappa_\theta(\rho(\eta_n))} - C_2
$$
and so  the ``moreover'' part of Theorem~\ref{thm:transverse have proj models} implies positivity of the integral. 
\end{proof}

\subsection{Proof of Proposition~\ref{prop:constructing functions}}

We begin by bounding the partial Iwasawa cocycle in terms of the
Hilbert distance.

\begin{lemma}\label{lem:growth estimate on Iwasawa}
There exist $a,b>0$ such that, for every $\alpha\in\theta$,
$x\in\Fc_\theta$, and $\gamma\in\Gamma_0$,
$$
\omega_\alpha\bigl(B_\theta(\rho(\gamma),x)\bigr)
\leq
a\,\dist_\Omega(o,\gamma o)+b.
$$
\end{lemma}

\begin{proof}
For every $\alpha\in\theta$, there exists $C_\alpha>0$ such that
$$
\omega_\alpha\bigl(B_\theta(h,x)\bigr)
\leq
C_\alpha\,
\omega_\alpha\bigl(\kappa_\theta(h)\bigr)
$$
for every $h\in\Gsf$ and $x\in\Fc_\theta$; see, for instance,
\cite[Lemma 4.1]{KZ_vector}. Hence
$$
\omega_\alpha\bigl(B_\theta(\rho(\gamma),x)\bigr)
\leq
C_\alpha\norm{\omega_\alpha}
\norm{\kappa_\theta(\rho(\gamma))}.
$$
Since $\theta$ is finite, the constants on the right-hand side may be
chosen uniformly in $\alpha$. The desired estimate now follows from the
``moreover'' part of
Theorem~\ref{thm:transverse have proj models}.
\end{proof}

Let $d$ be the dimension of the ambient vector space containing
$\Omega$, so that $\Omega\subset\Pb(\Rb^d)$, and set
$$
\lambda:=d+a,
$$
where $a$ is given by
Lemma~\ref{lem:growth estimate on Iwasawa}. For every
$\alpha\in\theta$ and $v\in\Gc$, define
$$
P_\alpha(v)
:=
\sum_{\gamma\in\Gamma_0}
e^{-\lambda\dist_\Omega(\gamma o,\pi(v))}
e^{\omega_\alpha\left(
B_\theta(\rho(\gamma)^{-1},\xi(v^+))
\right)},
$$
where $\pi:T^1\Omega\to\Omega$ is the basepoint projection.

We first verify that this series converges locally uniformly. Let
$L\subset\Gc$ be compact and set
$$
R_L:=\sup_{v\in L}\dist_\Omega(o,\pi(v)).
$$
For $v\in L$, the triangle inequality and
Lemma~\ref{lem:growth estimate on Iwasawa}, applied to $\gamma^{-1}$,
give
\begin{align*}
&e^{-\lambda\dist_\Omega(\gamma o,\pi(v))}
e^{\omega_\alpha\left(
B_\theta(\rho(\gamma)^{-1},\xi(v^+))
\right)}
\\
&\qquad\leq
e^{\lambda R_L+b}
e^{-(\lambda-a)\dist_\Omega(o,\gamma o)}
=
e^{\lambda R_L+b}
e^{-d\dist_\Omega(o,\gamma o)}.
\end{align*}
By Theorem~\ref{thm:tholozan},
$$
\sum_{\gamma\in\Gamma_0}
e^{-d\dist_\Omega(o,\gamma o)}
<+\infty.
$$
It follows that $P_\alpha$ is positive and continuous.

We next record the equivariance of $P_\alpha$. If $\eta\in\Gamma_0$,
then the cocycle identity, the $\rho$-equivariance of $\xi$, and the
change of variables $\gamma\mapsto\eta^{-1}\gamma$ give
\begin{align*}
&e^{\omega_\alpha\left(
B_\theta(\rho(\eta),\xi(v^+))
\right)}
P_\alpha(\eta v)
\\
&=
\sum_{\gamma\in\Gamma_0}
e^{-\lambda\dist_\Omega(\gamma o,\eta\pi(v))}
e^{\omega_\alpha\left(
B_\theta(\rho(\gamma)^{-1},\xi(\eta v^+))
+
B_\theta(\rho(\eta),\xi(v^+))
\right)}
\\
&=
\sum_{\gamma\in\Gamma_0}
e^{-\lambda\dist_\Omega(\eta^{-1}\gamma o,\pi(v))}
e^{\omega_\alpha\left(
B_\theta(\rho(\eta^{-1}\gamma)^{-1},\xi(v^+))
\right)}
\\
&=
P_\alpha(v).
\end{align*}
Thus
\begin{equation}\label{eqn:equiv of P}
e^{\omega_\alpha\left(
B_\theta(\rho(\eta),\xi(v^+))
\right)}
P_\alpha(\eta v)
=
P_\alpha(v).
\end{equation}

Since the basepoint of $g_tv$ is at Hilbert distance $\abs{t}$ from the
basepoint of $v$, we also have
\begin{equation}\label{eqn:bounds on P}
e^{-\lambda\abs{t}}P_\alpha(v)
\leq
P_\alpha(g_tv)
\leq
e^{\lambda\abs{t}}P_\alpha(v)
\end{equation}
for every $v\in\Gc$ and $t\in\Rb$.

Since
$\{\omega_\alpha|_{\mfa_\theta}\}_{\alpha\in\theta}$ is a basis of
$\mfa_\theta^*$, there exists a unique continuous map
$$
\mathbf Q:\Gc\to \mfa_\theta
$$
such that for every $\alpha\in\theta$,
\begin{equation}\label{eqn:defining Q}
\omega_\alpha\bigl(\mathbf Q(v)\bigr)
=
\log P_\alpha(v).
\end{equation}

By \eqref{eqn:equiv of P}, we have
\begin{equation}\label{eqn:equivariance of Q}
\mathbf Q(\eta v)
=
\mathbf Q(v)
-
B_\theta\bigl(\rho(\eta),\xi(v^+)\bigr)
\end{equation}
for every $\eta\in\Gamma_0$ and $v\in\Gc$.

Choose $c_\theta>0$ such that for every $H\in\mfa_\theta$,
$$
\norm{H}
\leq
c_\theta
\max_{\alpha\in\theta}
\abs{\omega_\alpha(H)}.
$$
By \eqref{eqn:bounds on P},
$$
\abs{
\omega_\alpha\bigl(\mathbf Q(g_tv)-\mathbf Q(v)\bigr)
}
\leq
\lambda\abs{t}.
$$
Consequently,
\begin{equation}\label{eqn:flow bound Q}
\norm{\mathbf Q(g_tv)-\mathbf Q(v)}
\leq
C_0\abs{t},
\qquad
C_0:=c_\theta\lambda.
\end{equation}

We now define
$$
\mathbf f(v)
:=
\mathbf Q(v)-\mathbf Q(g_1v).
$$
The map $\mathbf f$ is continuous, and
\eqref{eqn:flow bound Q} gives
$$
\sup_{v\in\Gc}\norm{\mathbf f(v)}
\leq
C_0.
$$
Moreover, since the $\Gamma_0$-action commutes with the geodesic flow
and $(g_1v)^+=v^+$, \eqref{eqn:equivariance of Q} gives
\begin{align*}
\mathbf f(\eta v)
&=
\mathbf Q(\eta v)-\mathbf Q(g_1\eta v)
\\
&=
\mathbf Q(v)
-
B_\theta\bigl(\rho(\eta),\xi(v^+)\bigr)
-
\mathbf Q(g_1v)
+
B_\theta\bigl(\rho(\eta),\xi((g_1v)^+)\bigr)
\\
&=
\mathbf f(v).
\end{align*}
Thus $\mathbf f$ is $\Gamma_0$-invariant.

It remains to verify the Cartan displacement estimate. First observe
that, for every $T\geq0$,
\begin{align*}
\int_0^T\mathbf f(g_tv)\,dt
&=
\int_0^T
\left(
\mathbf Q(g_tv)-\mathbf Q(g_{t+1}v)
\right)\,dt
=
\int_0^1\mathbf Q(g_tv)\,dt
-
\int_T^{T+1}\mathbf Q(g_tv)\,dt.
\end{align*}
Hence, by \eqref{eqn:flow bound Q},
\begin{equation}\label{eqn:integral versus Q}
\left\|
\int_0^T\mathbf f(g_tv)\,dt
-
\bigl(\mathbf Q(v)-\mathbf Q(g_Tv)\bigr)
\right\|
\leq
C_0.
\end{equation}

Now fix a compact set $K\subset\Gc$, and suppose that
$v\in K$ and $g_Tv\in\gamma K$. Then
$\gamma^{-1}g_Tv\in K$. Applying
\eqref{eqn:equivariance of Q} with $\eta=\gamma^{-1}$ gives
$$
\mathbf Q(\gamma^{-1}g_Tv)
=
\mathbf Q(g_Tv)
-
B_\theta\bigl(\rho(\gamma)^{-1},\xi(v^+)\bigr).
$$
It follows that
\begin{align}
&\left\|
\mathbf Q(v)-\mathbf Q(g_Tv)
+
B_\theta\bigl(\rho(\gamma)^{-1},\xi(v^+)\bigr)
\right\| =
\left\|
\mathbf Q(v)-\mathbf Q(\gamma^{-1}g_Tv)
\right\|
\leq
2\sup_{w\in K}\norm{\mathbf Q(w)}.
\label{eqn:compact Q error}
\end{align}

Set
$$
D
:=
1+
\operatorname{diam}
\left(
\{o\}\cup\pi(K)
\right),
$$
where the diameter is taken with respect to $\dist_\Omega$. Since
$\gamma^{-1}g_Tv\in K$,
$$
\dist_\Omega(\gamma o,\pi(g_Tv))
=
\dist_\Omega(o,\pi(\gamma^{-1}g_Tv))
<
D.
$$
Moreover, by Lemma~\ref{obs:distance between rays}, the point
$\pi(g_Tv)\in[\pi(v),v^+)$ is within distance $D$ of the ray
$[o,v^+)$. Therefore,
$$
B_\Omega(\gamma o,2D)\cap[o,v^+)
\neq\emptyset.
$$
By \cite[Lemma 7.3]{CZZ2024}, there exists $C_D>0$, depending only on
$D$, such that
\begin{equation}\label{eqn:shadow Iwasawa estimate}
\left\|
B_\theta\bigl(\rho(\gamma)^{-1},\xi(v^+)\bigr)
+
\kappa_\theta(\rho(\gamma))
\right\|
\leq
C_D.
\end{equation}

Combining \eqref{eqn:compact Q error} and
\eqref{eqn:shadow Iwasawa estimate}, we obtain
$$
\left\|
\mathbf Q(v)-\mathbf Q(g_Tv)
-
\kappa_\theta(\rho(\gamma))
\right\|
\leq
2\sup_{w\in K}\norm{\mathbf Q(w)}
+
C_D.
$$
Finally, \eqref{eqn:integral versus Q} yields
$$
\left\|
\int_0^T\mathbf f(g_tv)\,dt
-
\kappa_\theta(\rho(\gamma))
\right\|
\leq
C_0
+
2\sup_{w\in K}\norm{\mathbf Q(w)}
+
C_D.
$$
This proves the proposition with
$$
C_K
:=
C_0
+
2\sup_{w\in K}\norm{\mathbf Q(w)}
+
C_D. 
$$
\qed

\section{Smooth points of the Manhattan hypersurface}\label{sec:smooth points}

 Suppose that $\theta = \opp^* \theta$ and that $\Gamma < \Gsf$ is a non-elementary $\theta$-transverse group. Fix $\alpha_0 \in \theta$ and set $\theta_0 := \theta \smallsetminus \{\alpha_0\}$ so that we have a subspace $\mfa_{\theta_0}^* \subset \mfa_{\theta}^*$. We also set $\psi_0 := \sum_{\alpha \in \theta} \omega_{\alpha}$, which is positive on $\mfa_{\theta}^+$. Notice that 
 $$
 \mfa_{\theta_0}^* \oplus \Rb \cdot \psi_0 = \mfa_\theta^*.
 $$
 Since $\Qc_{\theta}(\Gamma)$ is convex and $\psi_0 > 0$ on $\mfa_{\theta}^+ \smallsetminus \{0\}$, there exists  a convex function $\Phi : \mfa_{\theta_0}^* \to \Rb$ so that 
$$
\partial \Qc_{\theta}(\Gamma)= \{ u + \Phi(u)\psi_0 : u \in \mfa_{\theta_0}^* \}.  
$$
In other words, for each $u \in \mfa_{\theta_0}^*$, $\Phi(u)$ is the abscissa of convergence, in the parameter $s$, of the series
$$
\sum_{\gamma \in \Gamma} e^{-(u + s \psi_0)(\kappa(\gamma))}.
$$
Since $\Phi$ is convex, it is differentiable almost everywhere, and we prove the following.

\begin{proposition} \label{prop:derivative}
  Suppose $\Phi$ is differentiable at $u \in \mfa_{\theta_0}^*$, $\phi := u + \Phi(u)\psi_0$, and $\mu_\phi$ is a $1$-dimensional $\phi$-Patterson--Sullivan measure for $\Gamma$ on $\Lambda_{\theta}(\Gamma)$.

There exists a set $A \subset \Lambda_\theta(\Gamma)$ with full $\mu_\phi$-measure with the following property: if $x \in A$, $\{\gamma_n\} \subset \Gamma$ converges to $x$ conically, and $\psi \in \mfa_{\theta_0}^*$, then 
$$
\lim_{n \rightarrow \infty} \frac{\psi(\kappa(\gamma_n))}{\psi_0(\kappa(\gamma_n))} = - \left. \frac{d}{dt} \right|_{t=0} \Phi( u +  t \psi).
$$
\end{proposition}

\begin{proof}  By the ergodic dichotomy (see \cite[Theorem 1.4]{CZZ2024}), the conical limit set either has full or zero $\mu_{\phi}$-measure. In the case when the conical limit set has zero measure, the proposition is vacuously true, so we can assume the conical limit set has full $\mu_{\phi}$-measure.

By replacing $\Gsf$ with a quotient, we can assume that  $\Gsf$ has trivial center   and $\Psf_\theta$ contains no simple factors of $\Gsf$, see \cite[Section 2.4]{CZZ2024}. Then, using Theorem~\ref{thm:transverse have proj models}, we can fix a projectively visible model $(\Omega, \Gamma_0, \rho, \xi)$ of $\Gamma$. 

By Proposition~\ref{prop:char of conical}, if $x \in \Lambda_\theta(\Gamma)$ and $\{\gamma_n\} \subset \Gamma$ converges to $x$ conically, then there exists $R > 0$ such that 
$$
x \in \bigcap_{n=1}^\infty \widehat \Oc_R(\gamma_n). 
$$

  By linearity, it suffices to prove the assertion when $\psi = \omega_{\alpha}$, $\alpha \in \theta_0$. Fix $\alpha \in \theta_0$ and write 
  $$
  r := - \left. \frac{d}{dt} \right|_{t=0} \Phi( u +  t \omega_{\alpha}) \quad \text{and} \quad \delta := \delta^{\omega_\alpha}(\Gamma)
  $$
  for simplicity. Since $\Gamma$ is $\{\alpha\}$-discrete for every $\alpha\in \theta$, by Lemma \ref{obs:finite critical exp fund weight}, $\delta < + \infty$. For $s \neq 0$, let 
$$
f(s) := \frac{\Phi(s \delta \omega_{\alpha} + u) - \Phi(u)}{s \delta} +r.
$$
  Then $f(s) \to 0$ as $s \to 0$. 
  
It also suffices to fix $\epsilon > 0$ and construct a full $\mu_\phi$-measure subset $A_\epsilon \subset  \Lambda_\theta(\Gamma)$ with the following property: if $x \in A_\epsilon$ and  $\{\gamma_j\} \subset \Gamma$ converges to $x$ conically, then 
\begin{equation} \label{eqn:ineq for r}
2\epsilon + \liminf_{j \rightarrow \infty} \frac{\omega_{\alpha}(\kappa(\gamma_j))}{\psi_0(\kappa(\gamma_j))} \geq r \geq -2\epsilon + \limsup_{j \rightarrow \infty} \frac{\omega_{\alpha}(\kappa(\gamma_j))}{\psi_0(\kappa(\gamma_j))}. 
\end{equation}
Indeed, to finish the proof we can take $A := \bigcap_{n \in \Nb} A_{1/n}$. 
  
  \medskip  
  Fix $\epsilon > 0$ and then fix $s > 0$ such that $\epsilon - f(s) > 0$ and $\epsilon + f(-s) > 0$. Let 
  $$\begin{aligned}
  \phi_s & := (s \delta \omega_{\alpha} + u) + \Phi(s \delta \omega_{\alpha} + u)\psi_0 \quad \text{and} \\
      \phi_{-s} & := (-s \delta \omega_{\alpha} + u) + \Phi(-s \delta \omega_{\alpha} + u)\psi_0. \\
  \end{aligned}
  $$
  Let $\mu$ be a $\delta$-dimensional $\omega_{\alpha}$-Patterson--Sullivan measure, let $\mu_{\phi_s}$ be a $1$-dimensional $\phi_s$-Patterson--Sullivan measure, and let $\mu_{\phi_{-s}}$ be a $1$-dimensional $\phi_{-s}$-Patterson--Sullivan measure for $\Gamma$ on $\Lambda_{\theta}(\Gamma)$.   Fix $R_0 \in \Nb$ sufficiently large so that the Shadow Lemma (Proposition~\ref{prop:Shadow Lemma}) holds for $\mu$, $\mu_{\phi}$,  $\mu_{\phi_s}$, and $\mu_{\phi_{-s}}$ whenever $R \geq R_0$. 

We first show that for $\mu_\phi$-a.e.\ $x \in \Lambda_{\theta}(\Gamma)$, if $\{\gamma_j\} \subset \Gamma$ converges to $x$ conically,  then 
$$
\liminf_{j \rightarrow \infty} \frac{\omega_{\alpha}(\kappa(\gamma_j))}{\psi_0(\kappa(\gamma_j))} \geq r-2\epsilon.
$$

  Fix $R \geq R_0$. For each $n \in \Nb$, recall 
$$
S_n^{\psi_0} = \{ \gamma \in \Gamma : \psi_0( \kappa(\gamma)) \in [n, n+ 1)\}.
$$
and let 
$$
E_{R,n} := \bigcup \left\{ \widehat\Oc_R(\gamma) : \gamma \in S_n^{\psi_0} \text{ and } \mu\left(\widehat\Oc_R(\gamma)\right) \ge e^{-(r - \epsilon) \delta \psi_0(\kappa(\gamma))} \right\}.
$$
Then by definition, 
$$
\mathbf{1}_{E_{R, n}} \le \sum_{\gamma \in S_n^{\psi_0}} \mu \left(\widehat\Oc_R(\gamma) \right)^s  e^{s(r - \epsilon) \delta \psi_0(\kappa(\gamma))} \mathbf{1}_{\widehat\Oc_R(\gamma)}
$$
where $\mathbf{1}$ denotes the indicator function, and hence 
$$
\mu_{\phi}(E_{R,n}) \le \sum_{\gamma \in S_n^{\psi_0}} \mu \left(\widehat{\Oc}_R(\gamma) \right)^s  e^{s(r - \epsilon) \delta \psi_0(\kappa(\gamma))} \mu_{\phi}\left(\widehat\Oc_R(\gamma) \right).
$$
Then, by the Shadow Lemma (Proposition~\ref{prop:Shadow Lemma})
$$\begin{aligned}
\mu_{\phi}(E_{ R,n}) & \ll \sum_{\gamma \in S_n^{\psi_0}} e^{-s \delta \omega_{\alpha}(\kappa(\gamma))} e^{s(r - \epsilon) \delta \psi_0(\kappa(\gamma))} e^{-\phi(\kappa(\gamma))} \\
& \asymp e^{s (r - \epsilon) \delta n - \Phi(u) n} \sum_{\gamma \in S_n^{\psi_0}} e^{-\left(s \delta \omega_{\alpha} + u\right) (\kappa(\gamma)) }.
\end{aligned} 
$$
Here and below, the implicit constants depend on $R$ but not on $n$. 

Notice that by the Shadow Lemma (Proposition~\ref{prop:Shadow Lemma}) and the uniform multiplicity of shadows in the sets $S_n^{\psi_0}$ (Lemma~\ref{lem:uniform multiplicity}),  
$$
\sum_{\gamma \in S_n^{\psi_0}} e^{- \left( (s \delta \omega_{\alpha} + u) + \Phi(s \delta\omega_{\alpha} + u)\psi_0\right)(\kappa(\gamma))} =\sum_{\gamma \in S_n^{\psi_0}} e^{- \phi_s(\kappa(\gamma))}\asymp \sum_{\gamma \in S_n^{\psi_0}} \mu_{\phi_s}\left(\widehat\Oc_R(\gamma)\right) \ll 1.
$$
Hence,
$$
\sum_{\gamma \in S_n^{\psi_0}} e^{-\left(s \delta \omega_{\alpha} + u\right) (\kappa(\gamma)) } \ll e^{\Phi(s \delta \omega_{\alpha} + u) n},
$$
and therefore
$$
\mu_{\phi}(E_{R,n}) \ll e^{s (r - \epsilon) \delta n - \Phi(u) n} e^{\Phi(s \delta \omega_{\alpha} + u) n}.
$$
Then
$$
\mu_{\phi}(E_{R,n}) \ll \left( e^{s (r - \epsilon) \delta - r s \delta + f(s) s \delta} \right)^n = \left( e^{- s \delta(\epsilon - f(s))} \right)^n.
$$
Since $s > 0$ satisfies $\epsilon - f(s) > 0$, we have 
$$
\sum_{n = 1}^{\infty} \mu_{\phi}(E_{R,n}) < + \infty.
$$
By the Borel--Cantelli lemma, this implies that for $\mu_{\phi}$-a.e.\ $x \in \Lambda_{\theta}(\Gamma)$,
$$
x \notin E_{R,n} \quad \text{for all sufficiently large } n.
$$

Recall that $R \geq R_0$ was arbitrary. Let 
$$
E := \bigcup_{R\in [R_0,\infty) \cap \Nb} \bigcap_{m=1}^\infty \bigcup_{n=m}^\infty E_{R,n}.
$$
Then $E$ is $\mu_{\phi}$-null. The following claim finishes the proof of the first inequality in \eqref{eqn:ineq for r}.

\medskip

\noindent \textbf{Claim:} If $x \in \Lambda_{\theta}(\Gamma) \smallsetminus E$ and $\{\gamma_j\} \subset \Gamma$ converges to $x$ conically,  then 
$$
\liminf_{j \rightarrow \infty} \frac{\omega_{\alpha}(\kappa(\gamma_j))}{\psi_0(\kappa(\gamma_j))} \geq r-2\epsilon.
$$

\noindent \emph{Proof of Claim:} By Proposition~\ref{prop:char of conical}, there exists $R \in [R_0,\infty) \cap \Nb$ such that
$$
x \in \bigcap_{j \ge 1} \widehat\Oc_R(\gamma_j).
$$
Hence, in order to have $x \notin E_{R, n}$ for all sufficiently large $n \in \Nb$, we must have
$$
\mu\left(\widehat\Oc_R(\gamma_j)\right) < e^{-(r - \epsilon) \delta \psi_0(\kappa(\gamma_j))} \quad \text{for all sufficiently large } j \in \Nb.
$$
Then the Shadow Lemma (Proposition~\ref{prop:Shadow Lemma}) implies that 
$$
\frac{\omega_{\alpha}(\kappa(\gamma_j))}{\psi_0(\kappa(\gamma_j))} \ge r - 2\epsilon \quad \text{for all sufficiently large } j \in \Nb.
$$
This completes the proof of the claim. 

\medskip

A similar argument also shows the second inequality in \eqref{eqn:ineq for r}. We present this for the sake of completeness. Fix $R \geq R_0$ and consider the set
$$
F_{R,n} := \bigcup \left\{ \widehat\Oc_R(\gamma) : \gamma \in S_n^{\psi_0} \text{ and } \mu\left(\widehat\Oc_R(\gamma)\right) \le e^{-(r + \epsilon) \delta \psi_0(\kappa(\gamma))} \right\}.
$$
Then by definition, 
$$
\mathbf{1}_{F_{R, n}} \le \sum_{\gamma \in S_n^{\psi_0}} \mu\left(\widehat\Oc_R(\gamma)\right)^{-s}  e^{-s(r + \epsilon) \delta \psi_0(\kappa(\gamma))} \mathbf{1}_{\widehat\Oc_R(\gamma)}
$$
and so 
$$
\mu_{\phi}(F_{R,n}) \le \sum_{\gamma \in S_n^{\psi_0}} \mu\left(\widehat\Oc_R(\gamma)\right)^{-s}  e^{-s(r + \epsilon) \delta \psi_0(\kappa(\gamma))} \mu_{\phi}\left(\widehat\Oc_R(\gamma)\right).
$$

Then, by the Shadow Lemma (Proposition~\ref{prop:Shadow Lemma})
$$\begin{aligned}
\mu_{\phi}(F_{ R,n}) & \ll \sum_{\gamma \in S_n^{\psi_0}} e^{s \delta \omega_{\alpha}(\kappa(\gamma))} e^{-s(r + \epsilon) \delta \psi_0(\kappa(\gamma))} e^{-\phi(\kappa(\gamma))} \\
& \asymp e^{-s (r + \epsilon) \delta n - \Phi(u) n} \sum_{\gamma \in S_n^{\psi_0}} e^{-\left(- s \delta \omega_{\alpha} +  u\right) (\kappa(\gamma)) }.
\end{aligned} 
$$
Here and below, the implicit constants depend on $R$ but not on $n$. 

Notice that by the Shadow Lemma  (Proposition~\ref{prop:Shadow Lemma})  and the uniform multiplicity of shadows in the sets $S_n^{\psi_0}$ (Lemma~\ref{lem:uniform multiplicity}),
$$
\sum_{\gamma \in S_n^{\psi_0}} e^{- \left( (-s \delta \omega_{\alpha} + u) + \Phi(-s \delta\omega_{\alpha} + u)\psi_0\right)(\kappa(\gamma))} =\sum_{\gamma \in S_n^{\psi_0}} e^{- \phi_{-s}(\kappa(\gamma))}\asymp \sum_{\gamma \in S_n^{\psi_0}} \mu_{\phi_{-s}}\left(\widehat\Oc_R(\gamma)\right) \ll 1.
$$
Hence,
$$
\sum_{\gamma \in S_n^{\psi_0}} e^{-\left(- s \delta \omega_{\alpha} + u\right) (\kappa(\gamma)) } \ll e^{\Phi(- s \delta \omega_{\alpha} + u) n},
$$
and therefore
$$
\mu_{\phi}(F_{R,n}) \ll e^{- s (r + \epsilon) \delta n - \Phi(u) n} e^{\Phi(- s \delta \omega_{\alpha} + u) n}.
$$
Then
$$
\mu_{\phi}(F_{R,n}) \ll \left( e^{- s (r + \epsilon) \delta + r s \delta - f(-s) s \delta} \right)^n = \left( e^{- s \delta(\epsilon + f(-s))} \right)^n.
$$
Since $\epsilon + f(-s) > 0$,
$$
\sum_{n = 1}^{\infty} \mu_{\phi}(F_{R,n}) < + \infty.
$$
By the Borel--Cantelli lemma, this implies that for $\mu_{\phi}$-a.e.\ $x \in \Lambda_{\theta}(\Gamma)$,
$$
x \notin F_{R,n} \quad \text{for all sufficiently large } n.
$$

Recall that $R \geq R_0$ was arbitrary. Hence the set 
$$
F := \bigcup_{R\in [R_0,\infty) \cap \Nb} \bigcap_{m=1}^\infty \bigcup_{n=m}^\infty F_{R,n}
$$
is $\mu_{\phi}$-null. As in the preceding claim, one verifies that if $x \in \Lambda_{\theta}(\Gamma) \smallsetminus F$ and $\{\gamma_j\} \subset \Gamma$ converges to $x$ conically,  then 
$$
\limsup_{j \rightarrow \infty} \frac{\omega_{\alpha}(\kappa(\gamma_j))}{\psi_0(\kappa(\gamma_j))} \leq r+2\epsilon.
$$

\medskip

Then $A_\epsilon := \Lambda_{\theta}(\Gamma) \smallsetminus (E \cup F)$ is our desired full $\mu_{\phi}$-measure subset. 
\end{proof}

\section{Continuity properties of BMS measures}

Suppose $\theta = \opp^* \theta$, $\Gamma < \Gsf$ is $\theta$-transverse, and $(\Omega, \Gamma_0, \rho, \xi)$ is a projectively visible model  of $\Gamma$.  

We prove the following continuity result for BMS measures associated to linear forms that are positive on the $\theta$-limit cone and have a critical gap at infinity. 
  
  \begin{theorem} \label{thm:critical gap is stable}
    Suppose $\phi \in \mfa_\theta^*$ is positive on $\mathcal{L}_{\theta}(\Gamma) \smallsetminus \{0\}$ and has a critical gap at infinity for $\Gamma$. Then the following statements hold: 
    \begin{enumerate}
    \item There exists a neighborhood $\Uc \subset \mfa_{\theta}^*$ of $\phi$ such that, for every $\psi \in \Uc$,
    \begin{enumerate}
    \item[(1-a)] $\psi$ is positive on $\mathcal{L}_{\theta}(\Gamma) \smallsetminus \{0\}$ and
    \item[(1-b)] $\psi$ has a critical gap at infinity for $\Gamma$.
    \end{enumerate} 
    \item If $m_\psi$ denotes the BMS measure for $\psi \in \Uc$ and $f : \Gamma_0 \backslash \Gc \rightarrow \Rb$ is a bounded continuous function, then 
    $$
    \psi \in \Uc \mapsto \int f\,dm_\psi \in \Rb
    $$
    is continuous. 
    \end{enumerate}

  \end{theorem}

The rest of the section is devoted to the proof of the theorem. Fix $0 < \epsilon < 1$ such that 
   \begin{equation}\label{eqn:epsilon est 1}
 \left(\frac{1+\epsilon}{1-\epsilon} \right)\delta_\infty^\phi(\Gamma) <  \delta^\phi(\Gamma).
  \end{equation}
  Let $\mathbb{S}$ denote the unit sphere in $\mfa_\theta$ with respect to some norm. By compactness of $\mathbb{S} \cap \mathcal{L}_{\theta}(\Gamma)$ and positivity of $\phi$ on $\mathcal{L}_\theta(\Gamma) \smallsetminus \{0\}$, we can fix a neighborhood $\Uc$ of $\phi$ such that 
  $$
(1-\epsilon)  \phi(v) \leq \psi(v) \leq (1+\epsilon) \phi(v)
  $$
  for all $\psi \in \Uc$ and $v \in \mathbb{S} \cap \mathcal{L}_\theta(\Gamma)$. By linearity, these inequalities extend to all of $\mathcal{L}_\theta(\Gamma)$. Thus (1-a) follows. 
  Further,  
     \begin{equation}\label{eqn:epsilon est 2}
    \delta_\infty^\psi(\Gamma) \leq \left(\frac{1}{1-\epsilon} \right)  \delta_\infty^\phi(\Gamma) \quad \text{and} \quad \delta^\psi(\Gamma) \geq \left( \frac{1}{1+\epsilon} \right) \delta^\phi(\Gamma) 
    \end{equation}
    for all $\psi \in \Uc$. Then by our choice of $\epsilon > 0$, every linear form $\psi \in \Uc$ has a critical gap at infinity for $\Gamma$. This shows (1-b).
    
Now Theorems~\ref{thm:Wen's result} and~\ref{thm:consequence of finite volume} imply that for each $\psi \in \Uc$, the BMS measure $m_\psi$ has finite mass and the Hilbert geodesic flow on $\Gamma_0 \backslash \Gc$ is ergodic with respect to $m_\psi$.  

Since $\psi \mapsto \delta^\psi(\Gamma)$ is convex (see, e.g., \cite[Theorem 13.1]{CZZ2024}), the critical exponent depends continuously on the linear form. Then the construction of the BMS measures and uniqueness of Patterson--Sullivan measures (see Theorem~\ref{thm:consequence of erg dich}) imply that for any compactly supported continuous function $f : \Gamma_0 \backslash \Gc \rightarrow \Rb$, the map 
\begin{equation}\label{eqn:cont of meas}
\psi \in \Uc \mapsto \int f\,dm_\psi
\end{equation}
is continuous. Thus, part (2) of the theorem holds for compactly supported continuous functions. To upgrade to bounded continuous functions, we use the following lemma. 

\begin{lemma}\label{after} After shrinking $\Uc$, for any $r > 0$ there exists a compact set $K \subset \Gamma_0 \backslash \Gc$ such that 
$$
\sup_{\psi \in \Uc} m_\psi(K^c) < r. 
$$
\end{lemma} 

Assuming the lemma, we complete the proof of (2) in the theorem. Fix $\psi_n \rightarrow \psi$ in $\Uc$. Then fix $r > 0$ and a compact set $K \subset \Gamma_0 \backslash \Gc$ such that 
$$
\sup_{\psi \in \Uc} m_\psi(K^c) < r. 
$$
Fix a compactly supported continuous function $\chi : \Gamma_0 \backslash \Gc \rightarrow [0,1]$ with $\chi|_{K} \equiv 1$. Then by the continuity in \eqref{eqn:cont of meas},
\begin{align*}
\limsup_{n \rightarrow \infty} & \abs{ \int f dm_{\psi_n} - \int f dm_\psi }   \leq \limsup_{n \rightarrow \infty} \abs{ \int \chi f dm_{\psi_n} - \int \chi f dm_\psi }  + 2 \norm{f}_\infty r \\
& = 2 \norm{f}_\infty r.
\end{align*}
Since $r > 0$ was arbitrary, $\int f dm_{\psi_n} \rightarrow \int f dm_\psi$. 
\qed

\medskip

It remains to prove the lemma. The argument is similar to the proof of \cite[Theorem 3.3]{Wen2026}, which is itself similar to the arguments in \cite[Section 5.3]{PS_finiteness}.  

\begin{proof}[Proof of Lemma~\ref{after}] Fix $R > 0$ sufficiently large and let 
    $$
    K_R := \{ v \in \Gc : \dist_\Omega(\pi(v),o) \leq R \}
    $$
 where $\pi : T^1 \Omega \rightarrow \Omega$ is the basepoint projection.   
Lemma~\ref{lem:monotone of entropy of K} implies that
 $$
 \lim_{R \rightarrow \infty} \delta^\phi(K_R) = \delta_\infty^\phi(\Gamma).
 $$
Using \eqref{eqn:epsilon est 1}, we can increase $R > 0$ and assume that 
  $$
 \left( \frac{1+\epsilon}{1-\epsilon}\right)  \delta^\phi(K_R) <  \delta^\phi(\Gamma).
  $$
Then \eqref{eqn:epsilon est 2} implies
   that for all $\psi \in \Uc$,
  \begin{equation}\label{eqn:gap with epsilon1}
\delta^\phi(K_R)  <  (1-\epsilon)\delta^\psi(\Gamma).
  \end{equation}

    Let $q : \Gc \rightarrow \Gamma_0 \backslash \Gc$ be the quotient map and let $\widehat{K}_R := q(K_R)$. Using the continuity in \eqref{eqn:cont of meas}, increasing $R$, and replacing $\Uc$ with a smaller neighborhood, we can assume that   for all $\psi \in \Uc$,
    $$
    m_\psi(\widehat{K}_R) > 0.
    $$

    For $\gamma \in \Gamma_{K_R}$,   let 
    $$
    A_\gamma := \{ v \in \Gc : \{ g_tv\}_{t \leq 0} \cap K_R \neq \emptyset \text{ and } \{g_tv\}_{t \geq 0} \cap \gamma K_R \neq \emptyset \}.
    $$
        Since the flow on $\Gamma_0 \backslash \Gc$ is ergodic with respect to each $m_\psi$ (Theorems~\ref{thm:Wen's result} and~\ref{thm:consequence of finite volume}), almost every flow line visits $\widehat{K}_R$ at an unbounded set of times. Hence if 
            $$
    \Ac := \bigcup_{\gamma \in \Gamma_{K_R}} A_\gamma, 
    $$
    then $q(\Ac)$ has full $m_\psi$-measure  for all $\psi \in \Uc$.
    
    \medskip

\noindent \textbf{Claim:} There exists $C =C(R) > 0$ such that: If $\gamma \in \Gamma_{K_R}$,  then 
$$
A_\gamma \subset \Lambda_\Omega(\Gamma_0) \times \Oc_{2R}(\gamma) \times [-C, C+ \dist_\Omega(o,\gamma o)]. 
$$

\noindent \emph{Proof of Claim:} We use the Hopf coordinates on $\Gc$. Let 
$$
C := \max \{ \abs{s} : (w^-, w^+, s) \in K_R\}.
$$
Then 
$$
\max\{ \abs{s} : (w^-, w^+, s) \in \gamma K_R\} \leq C + \dist_\Omega(o,\gamma o).
$$

Now fix $v = (v^-, v^+, t) \in A_\gamma$. By the definition of $A_\gamma$, we can choose $a \leq 0$ with $g_a v \in K_R$ and choose $b \geq 0$ with $g_b v \in \gamma K_R$. Then 
$$
t \geq t+a \geq -C
\quad\text{and}\quad 
t \leq t+b \leq C+ \dist_\Omega(o,\gamma o). 
$$
So 
$$ 
t \in [-C, C+ \dist_\Omega(o,\gamma o)]. 
$$
By the definition of $A_\gamma$, $[\pi(g_a v), v^+)$ intersects $B_\Omega(\gamma o, R)$ and $\pi(g_a v) \in B_\Omega(o,R)$, where $B_{\Omega}$ denotes a $\dist_{\Omega}$-ball. So Lemma~\ref{obs:distance between rays} implies that $[o,v^+)$ intersects $B_\Omega(\gamma o, 2R)$. Hence $v^+ \in  \Oc_{2R}(\gamma)$. This completes the proof of the claim. 
\medskip 

Notice that $\max_{v \in K_R} \norm{ \langle \xi(v^-), \xi(v^+) \rangle_\theta } $ is finite. So by the Shadow Lemma (Proposition~\ref{prop:Shadow Lemma}), if $\gamma \in \Gamma_{K_R}$, then 
$$
\tilde m_\psi(A_\gamma) \ll ( \dist_\Omega(o,\gamma o) + 2C) e^{-\delta^\psi(\Gamma) \psi(\kappa(\rho(\gamma)))}.
$$
Since $\phi$ is positive on $\mathcal{L}_{\theta}(\Gamma) \smallsetminus \{0\}$,  there exists $c_1 > 0$ such that  for all $\gamma \in \Gamma_0$,
  $$
\phi(\kappa(\rho(\gamma)))=\phi(\kappa_\theta(\rho(\gamma))) \geq c_1^{-1} \norm{\kappa_\theta(\rho(\gamma))}-c_1 .
  $$
 Then by Theorem~\ref{thm:transverse have proj models}, there exists $c_2, c_2' > 0$ such that for all $\gamma \in \Gamma_0$,
  $$
\dist_\Omega(o,\gamma o) \leq c_2 \phi(\kappa(\rho(\gamma)))+c_2'
 . $$
  We also have that for all $\gamma \in \Gamma_0$ and $\psi \in \Uc$,
 \begin{align*}
 \psi(\kappa(\rho(\gamma))) & \geq \phi(\kappa(\rho(\gamma))) - \norm{\phi-\psi}  \norm{\kappa_\theta(\rho(\gamma))} \\
 & \geq \left(1 -  c_1\norm{\phi-\psi}\right)  \phi(\kappa(\rho(\gamma)))-c_1^2  \norm{\phi-\psi}.
 \end{align*}
   So by replacing $\Uc$ with a smaller neighborhood, we can assume that 
   \begin{align*}
    \psi(\kappa(\rho(\gamma)))  \geq \left(1 -  \epsilon/2 \right)  \phi(\kappa(\rho(\gamma)))-1
 \end{align*}
    for all $\gamma \in \Gamma_0$ and $\psi \in \Uc$. 
    Since $\phi$ is positive on $\mathcal{L}_{\theta}(\Gamma) \smallsetminus \{0\}$ and thus the set 
    $$
    \{ \phi(\kappa(\rho(\gamma))) : \gamma \in \Gamma_0\}
    $$
    is bounded below, we obtain that for all
 $\gamma \in \Gamma_0$   and  $\psi \in \Uc$,  
  $$
  \tilde m_\psi(A_\gamma) \ll (  \phi(\kappa(\rho(\gamma)))+(2C + c_2')/c_2) e^{-\delta^\psi(\Gamma) (1-\epsilon/2) \phi(\kappa(\rho(\gamma)))} .
$$  
 Hence 
      $$
      \tilde m_\psi(A_\gamma) \ll e^{-\delta^\psi(\Gamma) (1-\epsilon) \phi(\kappa(\rho(\gamma)))}.
$$  
    Choose $s$ such that
    $$
    \delta^\phi(K_R) < s < \left(\frac{1-\epsilon}{1+\epsilon}\right) \delta^\phi(\Gamma).
    $$
    By \eqref{eqn:epsilon est 2}, we have $s < (1-\epsilon)\delta^\psi(\Gamma)$ for every $\psi \in \Uc$. Since the set $\{\phi(\kappa(\rho(\gamma))) : \gamma \in \Gamma_0\}$ is bounded below, the preceding estimate implies, with a constant independent of $\psi \in \Uc$, that
    $$
    \tilde m_\psi(A_\gamma) \ll e^{-s\phi(\kappa(\rho(\gamma)))}.
    $$
    The series $\sum_{\gamma \in \Gamma_{K_R}} e^{-s\phi(\kappa(\rho(\gamma)))}$ converges by the definition of $\delta^\phi(K_R)$. Hence the tails of
    $$
    \sum_{\gamma \in \Gamma_{K_R}} \tilde m_\psi(A_\gamma)
    $$
    tend to zero uniformly in $\psi \in \Uc$. Therefore, we can fix a finite set $F \subset \Gamma_{K_R}$ such that 
         for all $\psi \in \Uc$,  $$
    \sum_{\gamma \in \Gamma_{K_R} \smallsetminus F}  \tilde m_\psi(A_\gamma)  < r .
    $$
 Then 
    $$
    K := q\left( \bigcup_{\gamma \in F} A_\gamma \right)
    $$
    is a compact subset of $\Gamma_0 \backslash \Gc$ and 
    $$
\sup_{\psi \in \Uc} m_\psi(K^c) \leq  \sum_{\gamma \in \Gamma_{K_R} \smallsetminus F}  \tilde m_\psi(A_\gamma)  < r.
    $$
\end{proof}

\section{Proof of Theorem~\ref{thm:transverse main}}

We now combine our previous results to deduce Theorem~\ref{thm:transverse main}. Suppose  that $\theta = \opp^* \theta$ and that $\Gamma < \Gsf$ is a non-elementary $\theta$-transverse group. By replacing $\Gsf$ with a quotient, we can assume that  $\Gsf$ has trivial center   and $\Psf_\theta$ contains no simple factors of $\Gsf$, see \cite[Section 2.4]{CZZ2024}. Then, using Theorem~\ref{thm:transverse have proj models}, we can fix a projectively visible model $(\Omega, \Gamma_0, \rho, \xi)$ of $\Gamma$. 

Fix $\phi_0 \in \partial \Qc_{\theta}(\Gamma)$ such that $\phi_0$ is positive on $\mathcal{L}_{\theta}(\Gamma) \smallsetminus \{0\}$ and has a critical gap at infinity for $\Gamma$. By Theorem~\ref{thm:critical gap is stable}, there exists an open neighborhood $\Uc$ of $\phi_0$ in $\mfa_{\theta}^*$ such that every $\psi \in \Uc$ is positive on $\mathcal{L}_{\theta}(\Gamma) \smallsetminus \{0\}$  and has a critical gap at infinity for $\Gamma$.

As in Section~\ref{sec:smooth points}, we can view  $\partial \Qc_{\theta}(\Gamma)$ as the graph of a convex function $\Phi : \mfa_{\theta_0}^* \to \Rb$.  Next fix an open set $\Dc \subset \mfa_{\theta_0}^*$ such that $\{ u + \Phi(u)\psi_0 : u \in \Dc\}$ is an open neighborhood of $\phi_0$ in $\partial \Qc_\theta(\Gamma)$ and is contained in $\Uc$. 

For $\phi \in \mfa_\theta^*$, let $\hat f_\phi : \Gamma_0 \backslash \Gc \rightarrow \Rb$ be the function in Section~\ref{sec:constructing a function}. By Proposition~\ref{prop:derivative}, Lemma \ref{obs:positive integral}, and Corollary~\ref{cor:limit of ratios}, if $u \in \Dc$,  $\Phi$ is differentiable at $u$, and $\phi := u + \Phi(u)\psi_0$, then 
$$
\left. \frac{d}{dt} \right|_{t=0} \Phi( u +  t \psi) = - \frac{\int \hat f_{\psi} dm_\phi}{\int \hat f_{\psi_0} dm_\phi}
$$
for all $\psi \in \mfa_{\theta_0}^*$. Theorem~\ref{thm:critical gap is stable} implies that the right-hand side depends continuously on $u$. Since $\Phi$ is convex, its derivative exists almost everywhere and agrees there with a continuous function. A standard convexity argument therefore implies that $\Phi$ is differentiable everywhere on $\Dc$.

It remains to prove the strict-convexity assertion when $\Gamma$ is Zariski dense. Choose a convex neighborhood $\mathcal V$ of $\phi_0$ whose closure is contained in $\Uc$, and restrict the boundary patch to $\mathcal V$. Suppose $\phi_1 \neq \phi_2$ lie in this patch, let $t \in (0,1)$, and set $\phi_t := t\phi_1+(1-t)\phi_2$. Since $\phi_t \in \Uc$, Theorem~\ref{thm:Wen's result} implies that $(\Gamma,\phi_t)$ is of divergence type. Hence \cite[Corollary 13.2]{CZZ2024} gives $\delta^{\phi_t}(\Gamma)<1$. The critical exponent is continuous on $\Uc$, so $\phi_t$ lies in the interior of $\Qc_\theta(\Gamma)$. Thus every nontrivial open segment between two points of the boundary patch lies in the interior, proving local strict convexity. 
\qed

\appendix 

\section{GPS systems and Wen's finiteness result}\label{sec:appendix on Wen}

In this appendix, we recall the definition of a GPS system, state Wen's finiteness result for BMS measures \cite{Wen2026}, and then explain why Wen's result applies in our context. 

\subsection{GPS systems and expanding cocycles}

In \cite{BCZZ_coarse}, Blayac, Canary, Zhu, and the third named author introduced the notion of a Gromov--Patterson--Sullivan (GPS) system, consisting of a convergence group action, two coarse cocycles, and an associated Gromov product. We briefly recall the special case of continuous GPS systems.

Suppose that $M$ is a compact metrizable space and $\Gamma < \mathsf{Homeo}(M)$ is a non-elementary convergence group.
Recall that a continuous cocycle $\sigma : \Gamma \times M \to \Rb$ is a continuous function satisfying 
$$
\sigma(\gamma_1 \gamma_2 , x) = \sigma(\gamma_1, \gamma_2 x) + \sigma(\gamma_2, x) \quad \text{for all } \gamma_1, \gamma_2 \in \Gamma \text{ and } x \in M.
$$

Let $M^{(2)} := \{ (x, y) \in M^2 : x \neq y \}$. 
Denoting by $\gamma^+, \gamma^- \in M$ the attracting and repelling fixed points, respectively, of a loxodromic $\gamma \in \Gamma$, a continuous cocycle $\sigma$ is called \emph{proper} if for any sequence $\{\gamma_n \} \subset \Gamma$ of distinct loxodromic elements such that $\{(\gamma_n^+, \gamma_n^-)\}$ is precompact in $M^{(2)}$, we have
$$
\sigma(\gamma_n, \gamma_n^+) \to + \infty \quad \text{as } n \to  \infty.
$$
We now recall the definition of continuous GPS systems from \cite{BCZZ_coarse}.

\begin{definition}
  A \emph{continuous GPS system} consists of a non-elementary convergence group $\Gamma < \mathsf{Homeo}(M)$, continuous proper cocycles $\sigma, \bar \sigma : \Gamma \times M \to \Rb$, and a continuous function $G : M^{(2)} \to \Rb$ such that 
  $$
  G(\gamma x, \gamma y) = G(x, y) + \bar\sigma(\gamma, x) + \sigma(\gamma, y) \quad \text{for all }\gamma \in \Gamma \text{ and } (x, y) \in M^{(2)}.
  $$
  The function $G$ is called a \emph{Gromov product}.
\end{definition}

Gerasimov \cite[Proposition 8.3.1]{Gerasimov_floyd} showed that $M$ can be used to compactify $\Gamma$. In particular, $\Gamma \sqcup M$ has a unique metrizable compact topology such that $\Gamma, M \hookrightarrow \Gamma \sqcup M$ are embeddings and the induced $\Gamma$-action on $\Gamma \sqcup M$ is a convergence action (see also \cite[Proposition 2.4]{BCZZ_coarse}). A metric on $\Gamma \sqcup M$ generating this topology is called a \emph{compatible metric}. 

Given a cocycle $\sigma : \Gamma \times M \to \Rb$, a function $\norm{\cdot}_{\sigma} : \Gamma \to \Rb$ is called a \emph{$\sigma$-magnitude} if for any $\epsilon > 0$, there exists $C > 0$ such that for any $\gamma \in \Gamma$ and $x \in M$ with $\dist(\gamma^{-1}, x) > \epsilon$, we have
$$
\norm{\gamma}_{\sigma} - C \le \sigma(\gamma, x) \le \norm{\gamma}_{\sigma} + C.
$$
When a $\sigma$-magnitude exists for a proper cocycle $\sigma$, we call $\sigma$ \emph{expanding}. The cocycles in a GPS system are expanding \cite[Proposition 3.4]{BCZZ_coarse}.

Given an expanding cocycle $\sigma$ on $M$ and $\delta \ge 0$, a probability measure $\mu$ on $M$ is a \emph{$\delta$-dimensional Patterson--Sullivan measure} for $\sigma$ if, for every $\gamma \in \Gamma$, the measure $\gamma_*\mu$ is absolutely continuous with respect to $\mu$ and 
$$
\frac{d \gamma_* \mu}{d\mu}(x) = e^{-\delta \sigma(\gamma^{-1}, x)} \quad \mu\text{-a.e.}
$$
The \emph{critical exponent} associated to $\sigma$ is given by 
$$
\delta^\sigma(\Gamma) := \inf \left\{ s > 0 : \sum_{\gamma \in \Gamma} e^{-s \norm{\gamma}_\sigma} < +\infty\right\} \in [0,+\infty]
$$
and $\sigma$ is of \emph{divergent type} if 
$$
 \sum_{\gamma \in \Gamma} e^{-\delta^\sigma(\Gamma) \norm{\gamma}_\sigma} = +\infty.
$$
We will mainly consider a Patterson--Sullivan measure whose dimension is equal to the critical exponent of the cocycle. 

One can also define an \emph{entropy at infinity} $\delta_\infty^\sigma(\Gamma)$ as in Section~\ref{subsec:critical gap} by replacing the exponent $\phi(\kappa(\gamma))$ with $\norm{\gamma}_\sigma$.

\subsection{BMS measures and Wen's finiteness result} 

Given an expanding cocycle $\sigma'$ on $M$, we can define a $\Gamma$-action on $M^{(2)} \times \Rb$ by 
$$
\gamma \cdot (x,y, s) = (\gamma x, \gamma y, s +\sigma'(\gamma, y)).
$$
This action is proper \cite[Proposition 10.2]{BCZZ_coarse}, and this space has a natural flow defined by 
$$
\psi^t( x,y,s) = (x,y,s+t)
$$
which descends to a flow on the quotient $\Gamma \backslash (M^{(2)} \times \Rb)$, also denoted by $\psi^t$. 

Now suppose $(\sigma, \bar \sigma, G)$ is a GPS system such that  $\delta^{\sigma}(\Gamma) = \delta^{\bar \sigma}(\Gamma) =: \delta$. Suppose also that 
$\mu, \bar\mu$ are Patterson--Sullivan measures of dimension $\delta$ for $\sigma$ and $\bar \sigma$, respectively.  Notice that the cocycle $\sigma'$ defining the action of $\Gamma$ on $M^{(2)} \times \Rb$ may differ from the cocycles in the GPS system.

The Radon measure $\tilde m_{\sigma}$ on $M^{(2)} \times \Rb$ defined by
$$
d \tilde m_{\sigma}(x,y,t) := e^{ \delta G(x,y)} d \bar\mu(x) d \mu(y) dt
$$
is invariant under the $\Gamma$-action and the flow $\{\psi^t\}$. Since $\Gamma$ acts properly on $M^{(2)} \times \Rb$, this measure descends to a flow-invariant measure $m_{\sigma}$ on the quotient $\Gamma \backslash (M^{(2)} \times \Rb)$ called a \emph{$(\sigma, \bar \sigma, G)$-BMS measure}.

  \begin{theorem}[{\cite[Theorems 4.8 and 4.9]{Wen2026}}] \label{thm:Wen's result 2} With the notation above, suppose in addition that: 
  \begin{itemize}
  \item $\delta_\infty^\sigma(\Gamma) < \delta^\sigma(\Gamma)  < +\infty$, 
  \item For every $\epsilon > 0$ there exists $C_\epsilon > 1$ such that   for all $\gamma \in \Gamma$,
  $$
  \norm{\gamma}_{\sigma'} \leq C_\epsilon e^{\epsilon \norm{\gamma}_\sigma}.
  $$

  \end{itemize} 
  Then $\sigma, \bar \sigma$ are of divergent type and the $(\sigma, \bar \sigma, G)$-BMS measure on $\Gamma \backslash (M^{(2)} \times \Rb)$ is finite. 
  \end{theorem} 

\subsection{Proof of Theorem~\ref{thm:Wen's result}} We now explain why Theorem~\ref{thm:Wen's result} is a consequence of Theorem~\ref{thm:Wen's result 2}. 

By \cite[Proposition 10.3]{BCZZ2024}, the cocycle
$$
\sigma(\gamma, x) := \phi \circ B_\theta( \rho(\gamma),\xi(x))
$$
on $\Lambda_\Omega(\Gamma_0)$ is part of a continuous GPS system with magnitude $\norm{\gamma}_\sigma := \phi(\kappa(\rho(\gamma)))$. 

Since $\Gamma_0$ is projectively visible, the action of $\Gamma_0$ on $\Lambda_\Omega(\Gamma_0)$ is a convergence group action; see \cite[Proposition 3.15]{CZZ2026}.
 By \cite[Lemma 3.2]{Bray_ergodicity}, 
$$
\sigma'(\gamma, x) := \beta_{x}(\gamma^{-1} o, o)
$$
defines a continuous cocycle for the action of $\Gamma_0$ on $\Lambda_\Omega(\Gamma_0)$. 

Theorem~\ref{thm:Wen's result} is a consequence of Theorem~\ref{thm:Wen's result 2} and the next two lemmas. 

\begin{lemma} $\sigma'$ is an expanding cocycle, and  $\norm{\gamma}_{\sigma'} := \dist_{\Omega}(o, \gamma o)$, $\gamma \in \Gamma_0$, is a $\sigma'$-magnitude. \end{lemma}

\begin{proof}   We first show that $\sigma'$ is proper. Suppose not. Then there exists a sequence $\{\gamma_n\} \subset \Gamma_0$ of distinct loxodromic elements (in the convergence group sense) such that two sequences $\{ \gamma_n^+\}$ and $\{ \gamma_n^-\}$ converge to two distinct points in $\Lambda_{\Omega}(\Gamma_0)$ as $n \to \infty$ while we have $\sup_{n \in \Nb} \beta_{\gamma_n^+}(o, \gamma_n o) < + \infty$.

  On the other hand, there exists $R > 0$ such that $\dist_{\Omega}(o, (\gamma_n^+, \gamma_n^-)) < R$ for all $n \in \Nb$. Since each $(\gamma_n^+, \gamma_n^-)$ is invariant under $\gamma_n$, we also have that $\dist_{\Omega}(\gamma_n o, (\gamma_n^+, \gamma_n^-)) < R$ for all $n \in \Nb$. This implies
  $$
  \abs{\beta_{\gamma_n^+}(o, \gamma_n o) - \dist_{\Omega}(o, \gamma_n o)} < 2R \quad \text{for all } n \in \Nb.
  $$
  Hence, the sequence $\{\gamma_n o \} \subset \Omega$ is bounded, which is a contradiction. This shows the properness of $\sigma'$.

Next we show that $\norm{\gamma}_{\sigma'} := \dist_{\Omega}(o, \gamma o)$ defines a $\sigma'$-magnitude. Fix a compatible metric $\dist$ on $\Gamma_0 \sqcup \Lambda_{\Omega}(\Gamma_0)$. Suppose, to the contrary, that there exist $\epsilon > 0$ and sequences $\{ \gamma_n \} \subset \Gamma_0$ and $\{x_n \} \subset \Lambda_{\Omega}(\Gamma_0)$ with $\dist(\gamma_n^{-1}, x_n) > \epsilon$ such that $\abs{ \beta_{x_n}(\gamma_n^{-1} o, o) - \dist_{\Omega}(o, \gamma_n o)} > n$ for all $n \in \Nb$.

  After passing to a subsequence, we may assume that $x_n \to x \in \Lambda_{\Omega}(\Gamma_0)$ and $\gamma_n^{-1} o \to y \in \Lambda_{\Omega}(\Gamma_0)$ as $n \to \infty$. Then $x \neq y$ by hypothesis. Since $\Gamma_0$ is projectively visible, $(x, y) \subset \Omega$, and hence $\sup_{n \in \Nb} \dist_{\Omega}(o, [\gamma_n^{-1} o, x_n]) < + \infty$.  This implies that $\sup_{n \in \Nb} \abs{\beta_{x_n} (\gamma_n^{-1} o, o) - \dist_{\Omega}(\gamma_n^{-1} o, o)} < + \infty$, which is a contradiction. This finishes the proof.
\end{proof} 

\begin{lemma} For every $\epsilon > 0$ there exists $C_\epsilon > 1$ such that   for all $\gamma \in \Gamma_0$,
  $$
  \norm{\gamma}_{\sigma'} \leq C_\epsilon e^{\epsilon \norm{\gamma}_\sigma}.
  $$
\end{lemma} 
  
  \begin{proof} Since $\phi$ is positive on $\mathcal{L}_{\theta}(\Gamma) \smallsetminus \{0\}$,  there exists $c_1 > 0$ such that 
  $$
  \norm{\gamma}_\sigma =\phi(\kappa_\theta(\rho(\gamma))) \geq c_1^{-1} \norm{\kappa_\theta(\rho(\gamma))}-c_1 
  $$
 for all $\gamma \in \Gamma_0$.  Then the ``moreover'' part of Theorem~\ref{thm:transverse have proj models} implies the desired estimate.  
  \end{proof}

\bibliographystyle{alpha}
\bibliography{geom}

\end{document}